%% file: root.tex
\newcommand{\thetitle}{Ordered Risk Minimization: Learning More from Less Data}

\documentclass[todo=true, usebiblatex=false, mode=cdc, color=true]{kulconf}  
\configuretemplate{P. Coppens}{\thetitle}

\usepackage[mode=arxiv]{arxivix}
\usepackage{tikzutils}
\usepackage{booktabs, multirow, tabularx}
\usepackage{array}
\usepackage[ruled, linesnumbered]{algorithm2e}

\newcommand{\major}[1]{\mathcal{M}^{#1}}
\newcommand{\bftab}{\fontseries{b}\selectfont}

\title{\thetitle}
\author{Peter Coppens and Panagiotis Patrinos$^\dagger$
\thanks{$^\dagger$P. Coppens and P. Patrinos are with the Department of Electrical
Engineering (ESAT-STADIUS), KU Leuven, Kasteelpark Arenberg
10, 3001 Leuven, Belgium.
        {Email: \tt\footnotesize peter.coppens@kuleuven.be, panos.patrinos@kuleuven.be}}%
\thanks{This work was supported by: the Research Foundation
Flanders (FWO) PhD grant 11E5520N and research projects G081222N, G033822N, G0A0920N;
European Union's Horizon 2020 research and innovation programme under the Marie Skłodowska-Curie grant agreement No. 953348.}%
}%

\begin{document}

\maketitle
\firstpage

\begin{abstract}
    We consider the worst-case expectation of a permutation invariant ambiguity set of discrete distributions
    as a proxy-cost for data-driven expected risk minimization. For this framework, we coin the term \emph{ordered risk minimization}
    to highlight how results from order statistics inspired the proxy-cost. Specifically, we show how such costs serve as 
    point-wise high-confidence upper bounds of the expected risk. The confidence level can be determined tightly for any sample size. 
    Conversely we also illustrate how to calibrate the size of the ambiguity set such that the high-confidence upper 
    bound has some user specified confidence. This calibration procedure notably supports $\phi$-divergence based ambiguity sets. 
    Numerical experiments then illustrate how the resulting scheme both 
    generalizes better and is less sensitive to tuning parameters compared to the empirical risk minimization approach. 
\end{abstract}

\begin{pub}
\end{pub}


\input{content/intro.tex}

\input{content/statistics.tex}
\input{content-old/cases.tex}

\printbibliography{}

\appendices
\input{content/preliminaries.tex}
\begin{arxiv}

\input{content-old/appendix/svm.tex}

\end{arxiv}

\end{document}

%% file: content/intro.tex
\section{Introduction}
The problem of \emph{expected risk minimization} is ubiquitous in machine learning and statistics \cite{Shalev-Shwartz2013,Vapnik1998}.
It is based on the idea that the quality of a model can be assessed by measuring its expected error, quantified by some loss function. The expectation 
should be evaluated with respect to the data-generating distribution. However, in practice, only samples are available. So the expectation 
needs to be replaced with a data-driven proxy, which aggregates the data. The common solution is \emph{empirical risk minimization}
or the \emph{sample average approach (SAA)}, where one takes an average over the losses at the sampled data points. 

Despite its advantages, the SAA often exhibits excessive sensitivity to the specific data realizations, particularly in high-dimensional settings \cite[\S8.H]{Royset2022}, 
leading to diminished generalization capabilities of the model. To address this, researchers have turned to \emph{Distributionally Robust Optimization (DRO)}, aiming to robustify against disparities 
between the empirical and true data-generating distributions.

In DRO, a worst-case expectation with respect to distributions in an \emph{ambiguity set} centered on the empirical distribution serves as a proxy for the true expected risk. 
This set can be based on the Wasserstein distance \cite{Esfahani2018b}, $\phi$-divergences \cite{Ben-Tal2013}, hypothesis tests \cite{Bertsimas2009b}, and others. 
See \cite{Rahimian2019,Lin2022} for recent surveys. However, DRO faces challenges when determining the ambiguity set's size. 
Current approaches rely on concentration inequalities \cite{Delage2010} or asymptotic bounds \cite{Ben-Tal2013}, ensuring that true distribution is contained within the ambiguity 
set with high probability. Unfortunately, this often results in conservatism, caused by either loose constants in the concentration 
inequalities or the shape of the ambiguity set. As an alternative,  
bootstrapping or cross validation techniques are often employed (cf. \cite{Esfahani2018b}).
These can be computationally expensive and lack statistical guarantees for finite samples, similarly to the asymptotic bounds. 
Such guarantees are a requirement in safety-critical applications like control (e.g. constraint tightening in tube-based MPC \cite{Lorenzen2017, Aolaritei2023}).

To address conservativeness issues, \cite{Duchi2021b,Lam2019} focus on bounding the expectation directly as an alternative 
to creating a confidence bound for the entire distribution. 
This mimics the focus on the expectation in the statistical learning framework of \cite[\S1]{Vapnik1998}. However, their bounds are 
asymptotic and therefore also lack strong statistical guarantees. In this paper, we take the first steps towards a finite-sample version of their scheme. 
To achieve this, we draw inspiration from results in \emph{order statistics} \cite{David2003} and \emph{stochastic orders} \cite{Shaked2007} to motivate 
the use of permutation invariant ambiguity sets. Notably, the $\phi$-divergences used in \cite{Duchi2021b,Lam2019,Van2021} produce a specific case. We coin the term \emph{ordered risk minimization}
to emphasize our statistical motivations. We demonstrate how to calibrate the ambiguity set's size to upper bound the true 
expectation with high probability, even when the true distribution does not fall within the ambiguity set. 
This probability serves as an intuitive tuning parameter.

The remainder of the paper continues as follows. We first present some notation, before moving on the the problem statement in \cref{sec:problem}. 
There the proxy costs we use are presented as well as the calibration problem. We present the statistical interpretation of these proxy costs 
as high confidence upper bounds in \cref{sec:statistics} and solve the calibration problem in \cref{sec:calibration}. Numerical experiments are 
then presented in \cref{sec:case-studies}.

\paragraph*{Notation} Let $\Re$ denote the reals\ilarxiv{{ }and $\eRe$ the extended reals}. 
For some convex function $\phi \colon \Re^n \to \ilarxiv{\eRe}\ilpub{\Re}$, let $\phi^*$ denote the convex conjugate\ilarxiv{, $\partial \phi$ the subgradient}\ilarxiv{{ }and
$\dom \phi$ its domain}. 
\ilarxiv{For a set $\set{X}$ let $\iota_{\set{X}}(x) = 0$ if $x \in \set{X}$ and $+\infty$ otherwise be the indicator function of $\set{X}$.}
For integers $a, b$ let $[a, b] = \{a, \dots, b\}$ and $[b] = \{1, \dots, b\}$. Let $[x]_+ = \max(0, x)$. 
For real vectors $x, y \in \Re^n$ we use $\< x, y\>$ to denote the Euclidean inner product\ilarxiv{{ }and $\one_n \in \Re^n$ is the vector of all ones}. 
For a cone $\set{K}$ let $\set{K}^\circ \dfn \{y \colon \<x, y\> \leq 0, \forall x \in \set{K}\}$ denote its polar cone. 
Let $\Pi^n$ denote the permutations of $[n]$ (i.e., all bijections $[n] \to [n]$). We write $\pi x = (x_{\pi(1)}, \dots, x_{\pi(n)})$
for $\pi \in \Pi^n$, $x \in \Re^n$ and 
similarly let $\Pi^n y = \{\pi y \colon \pi \in \Pi^n\}$ denote the orbit of $y$ under $\Pi^n$.
Let $\Re^n_{\uparrow} \dfn \{x \colon x_{1} \leq x_2 \leq \dots \leq x_n\}$ 
denote the monotone cone and $\major{n}$ its polar (cf. \cref{prop:dual-monotone-cone} and \cref{eq:majorization-cone}). 
Let $\Delta^n \dfn \{\mu \colon \sum_{i=1}^n \mu_i = 1, \mu_i \geq 0, i \in [n]\}$ denote the probability simplex.
For a vector $x \in \Re^n$ let $x_{(1)} \leq x_{(2)} \leq \dots \leq x_{(n)}$ be the increasing permutation of the elements of $x$
with $x_{\uparrow} = (x_{(1)}, x_{(2)}, \dots, x_{(n)})$. 
For sets $A$, $B$ let $A + B \dfn \{a + b \colon a \in {A}, b \in {B}\}$
denote the Minkowski sum. 
For random variables $X, Y$ we write $X \deq Y$ to say $X$ is identically distributed to $Y$. Let $\esssup[X]$ 
denote the essential supremum. Let $X \deq \mathrm{U}[\ell, u]$ imply $X$ is uniformly distributed over $[\ell, u]$. For a set $A$, $X \deq \mathrm{U}[A]$ is then
uniformly distributed over $A$. Finally let $X \deq \mathcal{N}(\mu, \Sigma)$ 
denote $X$ that is normally distributed with mean $\mu$ and covariance $\Sigma$.

\section{Problem Statement} \label{sec:problem}
We consider expected risk minimization 
\begin{equation} \label{eq:erm}
    \minimize_{\theta \in \Theta} \quad \E[\ell(\theta, \xi)].
\end{equation}
Here $\Theta \subseteq \Re^{n_\theta}$ and $\xi \colon \Omega \to \Xi \subseteq \Re^{n_\xi}$ is a random vector on a 
probability space $(\Omega, \F, \prob)$. As is common \cite{Vapnik1998,Shalev-Shwartz2013}, we then assume access to 
\emph{independent and identically distributed (iid) samples} $\xi^{(1)}, \dots, \xi^{(n-1)}$. Let $\ell_i(\theta) = \ell(\theta, \xi^{(i)})$ for $i=1, \dots, n-1$ and 
take $\ell_n(\theta)$ such that it upper bounds the cost almost surely. That is $\prob[\ell(\theta, \xi) \leq \ell_n(\theta)] = 1$.
It is assumed that $\ell_n(\theta)$ is finite for any $\theta$. We do so for two reasons: \emph{(i)} an assumption on the tail of the distribution of 
$\ell(\theta, \xi)$ is required to find a confidence interval for the mean \cite{Bahadur1956}; \emph{(ii)} the scheme is simplified considerably. 
An example of a valid bound is $\ell_n(\theta) = \sup_{\xi \in \Xi} \, \ell(\theta, \xi)$. 

We will use a data-driven proxy for the expectation by introducing \emph{permutation invariant ambiguity sets}.
These are subsets $\amb$ of the probability simplex $\Delta^n$ such that, for each $\mu \in \amb$ any permutation of $\mu$ is also in $\amb$. 
The \emph{ordered risk minimization} problem is then: 
\begin{equation} \label{eq:ordered-risk}
    \minimize_{\theta \in \Theta} \quad \sup_{\mu \in \amb} \, \sum_{i=1}^{n} \mu_i \ell_i(\theta). 
\end{equation}
This proxy cost interpolates between the robust case for $\amb = \Delta^n$
and the sample average (including a term associated with $\ell_n(\theta)$) when $\amb = \{\one_n/n\}$.
The interpolation interpretation is also common in DRO \cite{Ben-Tal2013}. To find a good balance, our goal is to select $\amb$ such that, for all $\theta \in \Theta$,
\begin{equation} \label{eq:calibration}
    \prob\left[ \sup_{\mu \in \amb} \, \sum_{i=1}^{n} \mu_i \ell_i(\theta) \geq \E[\ell(\theta, \xi)] \right] \geq 1 - \delta,
\end{equation}
We refer to this problem as the \emph{calibration problem}. It robustifies against disparities between the empirical and true distributions. 
The parameter $\delta$ then serves as an intuitive, user-determined parameter that controls the conservativeness of the method. However, as illustrated by experiments, 
our method is relatively insensitive to the value of $\delta$. 

To find an ambiguity set $\amb$ satisfying \cref{eq:calibration} we need to somehow parametrize it. 
A well known class of permutation invariant ambiguity sets uses $\phi$-divergences \cite{Ben-Tal2013}. 
Let $\phi \colon \Re_+ \to \Re$ be lower semicontinuous, convex and $\phi(1) = 0$. 
Also, let\footnote{We take the 
lower semicontinuous envelope of the terms inside the sum \cite[Def.~6]{Chouzenoux2019} to handle cases where $\nu_i$ equals zero.}
$I_{\phi}(\mu, \nu) \dfn \ssum_{i=1}^{n} \nu_i \phi(\mu_i / \nu_i)$ for all $\mu, \nu \in \Delta^n$. 
A \emph{(centered) $\phi$-divergence ambiguity set} is then
\begin{equation} \label{eq:ambiguity-phi-div}
    \amb_{\alpha} \dfn \left\{ \mu \in \Delta^{n} \colon I_{\phi}\left(\mu, \frac{\one_{n}}{n}\right) = \sum_{i=1}^{n}  \frac{\phi(n \mu_i)}{n} \leq \alpha \right\}.
\end{equation}
In this work we consider two examples: \emph{total variation (TV)} for which $\phi(t) = |t - 1|$ and 
\emph{Kullback Leibler (KL)} divergence for which $\phi(t) = t \log t - t + 1$. However, our method works for any divergence. See \cite{Ben-Tal2013} for more examples.

To calibrate $\amb_\alpha$ the radius $\alpha \in \Re$ should then be the smallest value such that \cref{eq:calibration} still holds. We also provide an alternative parametrization, 
related to a well known bound by Anderson \cite{Anderson1969} and the conditional value-at-risk. 


It is important to note that the constraint in \cref{eq:calibration} is less stringent compared to DRO, which guarantees 
that the supremum in \cref{eq:calibration} acts as a high-confidence upper bound, uniformly over $\theta$\footnote{%
The mean bound will be uniform when 
\begin{equation*}
    \prob\left[ \sup_{\mu \in \amb} \, \sum_{i=1}^{n} \mu_i \ell_{i}(\theta) \geq \E[\ell(\theta, \xi)],\, \forall \theta \in \Theta \right] \geq 1-\delta.
\end{equation*}%
}. 
After all, we never require that the true distribution is contained within $\amb$ (as is the case in \cite{Beck2017}). 
The gap between the point-wise \cref{eq:calibration}
and the uniform equivalent is examined for $\phi$-divergences in \cite{Duchi2021b,Lam2019} in the asymptotic regime. 

We numerically approximate the calibration problem without samples from $\xi$. 
So the parameters of the set $\amb$ only need to be computed once and can be tabulated afterwards. This contrasts the complex derivations
and the resulting conservative constants associated with analytical approaches used to compute the radius of an 
ambiguity set in DRO \cite{Delage2010,Ben-Tal2013,Esfahani2018b}. 
We show experimentally how our calibration of $\amb$ according to \cref{eq:calibration} greatly improves generalization.

%% file: content/statistics.tex
\section{Statistical Framework} \label{sec:statistics}
The analysis of this section investigates upper bounds for the mean of a scalar \emph{random variable (rv)} $Z \colon \Omega \to \Re$, defined on some probability space $(\Omega, \F, \prob)$. 
These findings remain applicable to the previous section, when considering $Z = \ell(\theta, \xi)$ for fixed $\theta$. Let $F(z) = \prob[Z \leq z]$ 
denote the \emph{cumulative distribution function (cdf)} of $Z$. 
We assume access to iid samples $Z_{1}, \dots, Z_{n-1}$ and an upper bound denoted as $Z_{n} = \esssup[Z]$ for notational convenience, 
which satisfies $F(Z_{n}) = 1$. 

We then introduce the \emph{coverages}
\begin{equation} \label{eq:coverages}
    W_i = F(Z_{(i)}) - F(Z_{(i-1)}), \quad \forall i \in [n],
\end{equation}
where $-\infty = Z_{(0)} \leq Z_{(1)} \leq Z_{(2)} \leq \dots \leq Z_{(n)}$ denotes
an increasing permutation of $Z_1, \dots, Z_{n}$ called the \emph{order statistics} of $Z$ and
with $Z_{(0)}$ added for convenience.

The coverages will be used to bound the expectation. To do so, we require the following cone in $\Re^n$:
\begin{equation} \label{eq:majorization-cone}
    \major{n} = \left\{ x \colon \sum_{i=1}^k x_i \geq 0, \forall k \in [n-1], \sum_{i=1}^n x_i = 0 \right\}. 
\end{equation}
In \cref{prop:dual-monotone-cone} we prove that it corresponds to the polar of the monotone cone $\Re^n_{\uparrow}$. It has a history in isotonic regression \cite{Barlow1972},
majorization \cite{Steerneman1990} and also describes a stochastic order between discrete distributions \cite[p.~4]{Shaked2007} as we illustrate later.

The expectation bound is then as follows:
\begin{proposition} \label{prop:mean-bound-main}
    Suppose that $\amb \subseteq \Delta^n$ is permutation invariant\footnote{%
        Specifically, for any $\pi \in \Pi^n$ and $\mu \in \amb$, $\pi \mu = (\mu_{\pi(1)}, \dots, \mu_{\pi(n)}) \in \amb$. Moreover we assume sets $\amb$ are Lebesgue measurable.}.
    Take $W$ as in \cref{eq:coverages}. 
    Then, for any rv $Z \colon \Omega \to \Re$ with iid samples $\{Z_i\}_{i=1}^{n-1}$ and $Z_n = \esssup[Z] < +\infty$, 
    \begin{equation*}
        \prob\left[ \sup_{\mu \in \amb} \, \sum_{i=1}^{n} \mu_i Z_{i} \geq \E[Z] \right] \geq \prob[W \in \amb + \major{n}]. 
    \end{equation*}
\end{proposition}
\begin{proof}
    We split up the expectation\footnote{For details on the integral notation see \cite[Eq.~17.22]{Billingsley1995}.} as follows:
    \begin{align}
        \E[Z] = \int_{-\infty}^{Z_{(n)}} z \di F(z) = \sum_{i=1}^{n} \int_{Z_{(i-1)}}^{Z_{(i)}} z \di F(z),
    \end{align}
    with $Z_{(1)} \leq \dots \leq Z_{(n-1)}$ the order statistics and $Z_{(n)} = \esssup[Z]$. 
    The first equality follows from $\prob[Z \leq Z_{(n)}] = 1$ and the second from \cite[Thm.~16.9]{Billingsley1995}. 
    For each term
    $\int_{Z_{(i-1)}}^{Z_(i)} z \di F(z) \leq Z_{(i)} \int_{Z_{(i-1)}}^{Z_(i)} \di F(z) = Z_{(i)} (F(Z_{(i)}) - F(Z_{(i-1)}))$. Hence
    $\E[Z] \leq \sum_{i=1}^{n} Z_{(i)} W_i$.

    Condition on $W \in \amb + \major{n}$.
    Then,
    \begin{equation*}
        \E[Z] \leq \sum_{i=1}^{n} W_i Z_{(i)} \leq \sup_{\mu \in \amb + \major{n}} \ssum_{i=1}^{n} \mu_i Z_{(i)}. 
    \end{equation*}
    The expression on the right can be simplified by noting that
    \begin{align*}
        {\sup}_{\mu \in \amb + \major{n}} {\ssum_{i=1}^{n}} \mu_i Z_{(i)} = {\sup}_{\mu \in \amb} \ssum_{i=1}^{n} \mu_i Z_{(i)},
    \end{align*}
    where we use $(Z_{(1)}, \dots, Z_{(n)}) \in \Re^n_{\uparrow}$, the definition of the polar cone and \cref{prop:dual-monotone-cone}, 
    which implies $\ssum_{i=1}^{n} s_i Z_{(i)} \leq 0$ for any $s \in \major{n}$ with equality for $s = 0$. 

    Finally, let $\pi \in \Pi^n$ be the permutation such that $Z_{\pi(i)} = Z_{(i)}$ for $i \in [n]$ and let $\pi^{-1}$ denote its inverse (which exists, since permutations are bijections). 
    Then
    \begin{align*}
        &\sup_{\mu \in \amb} \, \ssum_{i=1}^{n} \mu_i Z_{i} = {\displaystyle\sup_{\mu \in \amb}} \, \ssum_{i=1}^{n} \mu_{\pi^{-1}(i)} Z_{(i)} \\
        &\quad = \sup_{\pi \mu \in \amb} \, \ssum_{i=1}^{n} \mu_i Z_{(i)} \labelrel={step:permutation} {\displaystyle\sup_{\mu \in \amb}} \, \ssum_{i=1}^{n} \mu_i Z_{(i)},
    \end{align*} 
    where \ref{step:permutation} uses permutation invariance. Hence, we showed $\prob[\sup_{\mu \in \amb} \, \ssum_{i=1}^{n} \mu_i Z_{(i)} \geq \E[Z]] \geq \prob[W \in \amb + \major{n}]$.
\end{proof}

From \cref{prop:mean-bound-main}, it is clear that the distribution of $W$ is important. 
Interestingly, when $F$ is continuous, then $W$ is always uniformly distributed over $\Delta^n$ \cite[Thm.~8.7.4]{Wilks1964}. 
For general distributions however, we can still establish a type of \emph{stochastic order} between the two distributions using $\major{n}$:
\begin{lemma} \label{cor:coverages-stochastic-order}
    Take $W = (W_1, \dots, W_n) \in \Delta^n$ as in \cref{eq:coverages}. Then, for any (Lebesgue measurable) $\amb \subseteq \Delta^n$, 
    \begin{equation*}
        \prob[W \in \amb + \major{n}] \geq \prob[\nu \in \amb + \major{n}],
    \end{equation*}
    with $\nu \deq \mathrm{U}[\Delta^n]$ uniformly distributed over $\Delta^n$. 
    
    For continuous cdf we have $\prob[W \in \amb] = \prob[\nu \in \amb]$. 
\end{lemma}
\begin{proof}
    For continuous cdf we refer to \cite[Thm.~8.7.4]{Wilks1964}. For discontinuous cdf
    we first introduce a construction of the joint distribution of random vectors 
    $W'$ and $\nu'$, such that $W' \deq W$ and $\nu' \deq \nu$ (i.e., the marginals are as specified in the lemma). 
    For this construction, we show that $\nu' \in \amb + \major{n}$ implies $W' \in \amb + \major{n}$ almost surely.
    So $\prob[W' \in \amb + \major{n}] \geq \prob[\nu' \in \amb + \major{n}]$. 
    From $W' \deq W$ and $\nu' \deq \nu$ we then get the required result. 

    The construction starts by taking $U_i \deq \mathrm{U}[0, 1]$ as uniform random variables, for $i \in [n - 1]$, with 
    $U_{(1)} \leq U_{(2)} \leq \dots \leq U_{(n-1)} \leq U_{(n)} = 1$ the uniform order statistics.
    Let $\sum_{i=1}^{k} \nu_i' = U_{(k)}$ for $k \in [n]$. Then, by \cite[\S6.4]{David2003}, $\nu' \deq \mathrm{U}[\Delta^n]$ or $\nu' \deq \nu$. 
    Meanwhile, the \emph{quantile transform} \cite[Lem.~1.2.4(i)]{Reiss1989} states that $Z_i' = F^{-1}(U_i)$ has cdf $F$, where $F^{-1}$ is the quantile function of $Z$. 
    Since $F$ (and therefore $F^{-1}$) is nondecreasing, we can apply \cite[Lem.~1.2.1]{Reiss1989} to claim that $F^{-1}(U_{(i)})$ is distributed as the $i$-th order statistic $Z_{(i)}$. 
    So with $Z_{(i)}' = F^{-1}(U_{(i)})$ we take $W'$ analogously to \cref{eq:coverages}. From $Z_{(i)}' \deq Z_{(i)}$ we then have $W' \deq W$.

    Both marginals are related as, for $k \in [n]$, 
    \begin{equation} \label{eq:relationship-distributions}
        \begin{aligned}
        &\ssum_{i=1}^{k} W_i' = F(Z_{(k)}') \\
        &\quad= F(F^{-1}(U_{(k)})) = F(F^{-1}(\ssum_{i=1}^{k} \nu_i')), 
        \end{aligned}
    \end{equation}
    where the first equality follows by summing \cref{eq:coverages} for $i \in [k]$ with $W_i'$, $Z_{(i)}'$ in place of $W_i$, $Z_i$, 
    the second by construction of $Z_{(k)}'$ and the third by construction of $\nu'$.
    Note that, for general distributions, we have $F(F^{-1}(p)) \geq p$ for all $p \in [0, 1]$ \cite[Ex.~3.2]{Shorack2017}, 
    with strict inequality iff $p \in (0, 1)$ is not in the range of $F$. Applying this to \cref{eq:relationship-distributions} gives
    \begin{equation} \label{eq:stoch-dominance}
        \sum_{i=1}^{k} \nu_i' \leq \sum_{i=1}^k W_i', \, \forall k \in [n-1] \text{ and } \sum_{i=1}^n \nu_i' = \sum_{i=1}^{n} W_i'.
    \end{equation}
    Observe how \cref{eq:stoch-dominance} corresponds to a conic inequality under $\major{n}$. The inequalities are strict iff 
    $\sum_{i=1}^{k} \nu_i'$ is not in the range of $F$ (i.e., it lies in a discontinuous jump of $F$). In that sense, \cref{eq:stoch-dominance} 
    models the gap between $\mathrm{U}[\Delta^n]$ and the coverages $W$. 

    To complete the proof, assume that $\nu' \in \amb + \major{n}$. By definition of the Minkowski sum, this is equivalent to there 
    being some $\mu \in \amb$ such that $\nu - \mu \in \major{n}$ or,
    $\sum_{i=1}^{k} \mu_i \leq \sum_{i=1}^{k} \nu_i'$ for $k \in [n-1]$ and $\sum_{i=1}^{n} \mu_i = \sum_{i=1}^{n} \nu_i'$ (cf.\ \cref{eq:majorization-cone}).
    Thus, from \cref{eq:stoch-dominance}, we have $\sum_{i=1}^{k} \mu_i \leq \sum_{i=1}^{k} W_i'$ for $k \in [n-1]$ and $\sum_{i=1}^{n} \mu_i = \sum_{i=1}^{n} W_i'$. 
    By definition of $\amb + \major{n}$ and \cref{eq:majorization-cone} this shows that $W' \in \amb + \major{n}$. So, by our arguments at the start of the 
    proof, we showed the required result. 
\end{proof}


\begin{figure}
    \input{assets/onesided/onesideddominated.tex}
    \caption{Lower bounds of the cumulative distribution for nonnegative $Z$ with cdf $F$. 
    The tightest lower bound supported on the samples is depicted in blue, while a feasible lower bound is depicted in red. 
    The green area is the expectation.\vspace{-1em}} \label{fig:cdf-bounds}
\end{figure}
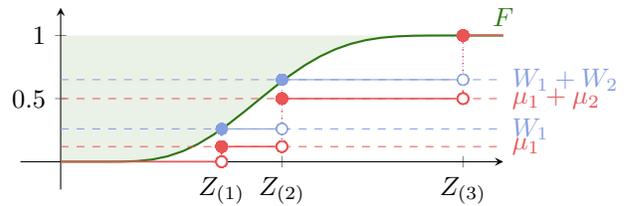

The result in \cref{cor:coverages-stochastic-order} can be interpreted in terms of lower bounding the cdf of $Z$. To illustrate this, we introduce 
the weighted empirical cdf
\begin{equation} \label{eq:weighted-cdf}
    F_\mu^n(x) = \ssum_{i=1}^{n} \mu_i \one_{[Z_{(i)}, +\infty)},
\end{equation}
where $\mu \in \Delta^n$ and $\one_{[Z_{(i)}, +\infty)}(z) = 1$ when $z \geq Z_{(i)}$ and zero otherwise. 
Note that $W \in \{\mu\} + \major{n}$ holds iff
\begin{equation*}
    F_{\mu}^n(Z_{(i)}) = \ssum_{i=1}^{k}\mu_i  \leq \ssum_{i=1}^{k} W_i = F(Z_{(i)}), \, \forall k \in [n],
\end{equation*}
with $\sum_{i=1}^{n} \mu_i = \sum_{i=1}^{n} W_i = 1$. This relationship relates to \cref{eq:stoch-dominance} and \cref{eq:majorization-cone}
and implies that $F_{\mu}^n$ should lower bound the cdf $F$ everywhere, as depicted in \cref{fig:cdf-bounds}. This inequality between cdfs is
the usual stochastic order \cite[\S1.A.1]{Shaked2007}. 
The event $W \in \amb + \major{n}$ is then equivalent to the existence of a cdf $F_\mu^n$
with weights $\mu \in \amb$ that lower bounds the true cdf. In terms of this interpretation, \cref{prop:mean-bound-main} 
follows from the fact that a lower bound on the cdf implies an upper bound on the expectation (cf. \cite{Anderson1969}, \cite[Eq.~1.A.5]{Shorack2009}).

By combining \cref{prop:mean-bound-main} and \cref{cor:coverages-stochastic-order} we directly prove
the main contribution of this paper:
\begin{theorem} \label{cor:mean-bound-main}
    Assuming the setting of \cref{prop:mean-bound-main} and taking $\nu \deq \mathrm{U}[\Delta^n]$, then
    \begin{equation*}
        \prob\left[ \sup_{\mu \in \amb} \, \sum_{i=1}^{n} \mu_i Z_{i} \geq \E[Z] \right] \geq \prob[\nu \in \amb + \major{n}]. 
    \end{equation*}
\end{theorem}

Note that, the important requirement of $\amb$ in \cref{cor:mean-bound-main} is that it is permutation invariant and a subset of 
the probability simplex. The resulting support functions are related to \emph{law-invariant, coherent risk measures} in literature \cite[\S6.3.5]{Shapiro2021}, \cite{Bertsimas2009b}, 
where the \emph{law} in our case is a permutation of the random vector. The $\phi$-based ambiguity sets considered below are the most frequently 
studied case of such risk measures. 

\section{Calibration Problem} \label{sec:calibration}
This section considers calibrating ambiguity sets $\amb$ such that \cref{eq:calibration} holds.
To do so, we first consider the $\phi$-divergence parametrization in \cref{eq:ambiguity-phi-div} and try to upper bound 
the smallest $\alpha$ such that \cref{eq:calibration} still holds for $\amb = \amb_{\alpha}$, which we denote as $\alpha_\star$.
Later we also provide an alternative parametrization similar to the conditional value-at-risk (cf.\ \cref{cor:simple-distortion} and the discussion below).

We can use the previous result to simplify \cref{eq:calibration}. 
Note that, for a fixed $\theta$, $\ell(\theta, \xi)$ is simply a scalar random variable, which we will denote as $Z$. So we consider
\begin{equation} \label{eq:calibration-scalar}
    \inf \left\{ \alpha \colon \prob\left[ \sup_{\mu \in \amb_{\alpha}} \sum_{i=1}^{n} \mu_i Z_{i} \geq \E[Z] \right] \geq 1 - \delta, \, \forall Z \right\},
\end{equation}
where we inherit the notation from the previous section. It's solution will upper bound $\alpha_\star$.
Note that all previous results were distribution-free. They hold for all (bounded) random variables $Z$, invariant of their underlying distribution. 
As such, by using \cref{cor:mean-bound-main}, the constraint in \cref{eq:calibration-scalar} can be conservatively approximated by 
\begin{equation}\label{eq:calibration-scalar:b}
    \prob\left[ \nu \in \amb_{\alpha} + \major{n} \right] \geq 1 - \delta, \quad \text{with } \nu \deq \mathrm{U}[\Delta^n].
\end{equation}
We use this to approximate the calibration problem:
\begin{proposition} \label{prop:divergence-bound-pre}
    Let $I_{\phi}$ denote the $\phi$-divergence, with $\amb_{\alpha}$ the associated ambiguity set as in \cref{eq:ambiguity-phi-div}. 
    Let $\alpha_\star$ denote the smallest $\alpha$ such that \cref{eq:calibration} holds for $\amb_\alpha$. Then
    \begin{align} \label{eq:quantile-function}
        \alpha_\star \leq \inf \left\{ \alpha \in \Re \colon \prob\left[ I^\diamond_{\phi}(\nu) \leq \alpha \right] \geq 1-\delta \right\},
    \end{align}
    with $\nu \deq \mathrm{U}[\Delta^n]$ and, for $\phi^*(s) = \sup_{t \geq 0} \, \{ts - \phi(t)\}$, 
    \begin{equation} \label{eq:div-adj}
        I^\diamond_{\phi}(\nu) \dfn \sup_{\lambda \in \Re^n_{\uparrow}} \left\{ \ssum_{i=1}^{n} \nu_i \lambda_i - \frac{1}{n} \phi^*(\lambda_i) \right\}.
    \end{equation}
\end{proposition}
\begin{proof}
    Note that $\nu \in \amb_{\alpha} + \major{n}$ holds iff there is some $\mu \in \amb_{\alpha}$ such that 
    $\nu - \mu \in \major{n}$. Equivalently, there should exist a $\mu \in \Delta^n$, which satisfies
    $I_{\phi}(\mu, \one_n/n) \leq \alpha$ and $\nu - \mu \in \major{n}$. We already showed that 
    the smallest $\alpha$ satisfying \eqref{eq:calibration-scalar:b} upper bounds $\alpha_\star$. The left-hand side of \eqref{eq:calibration-scalar:b} equals
    \begin{equation*}
        \prob\left[ 
            \inf_{\mu \in \Delta^n} \, \left\{\ssum_{i=1}^{n} \frac{1}{n} \phi(n \mu_i) \colon \nu - \mu \in \major{n} \right\} \leq \alpha \right].
    \end{equation*}
    Since $\nu \deq \mathrm{U}[\Delta^n]$, the constraint $\nu - \mu \in \major{n}$, together with $\sum_{i=1}^{n} \nu_i = 1$ 
    implies (cf.~\cref{eq:majorization-cone}) $\sum_{i=1}^{n} \mu_i = 1$. So the infimum can be taken over $\Re^n_+$. We place the constraint inside of the cost by noting that 
    the indicator function of a polar cone equals the support function of its dual \cite[Ex.~2.26]{Beck2017}. So we can rewrite \cref{eq:calibration-scalar:b} as follows:
    \begin{align*}
        &\inf_{\mu \in \Re^n_+} \, \sup_{\lambda \in \Re^n_{\uparrow}} \, \left\{\ssum_{i=1}^{n} \frac{1}{n} \phi(n \mu_i) + \<\lambda, \nu - \mu\> \right\}, \\
        & \leq  \sup_{\lambda \in \Re^n_{\uparrow}} \, \ssum_{i=1}^{n} \inf_{t_i \geq 0} \left\{ \phi(t_i) - t_i \lambda_i \right\} / n + \lambda_i \nu_i,
    \end{align*}
    The inequality follows by weak duality, the substitution $t_i = n\mu_i$ and separability.  
    The infima inside the sum are $-\phi^*(\lambda_i)$, completing the proof by the argument preceding \cref{eq:calibration-scalar:b}.
\end{proof}

Observe that the right-hand side of \cref{eq:quantile-function} is essentially the $1-\delta$ quantile of the scalar random variable $I^\diamond_{\phi}(\nu)$. 
Unfortunately its distribution is unknown. However, we can sample from it. To do so note that $\nu \deq \mathrm{U}[\Delta^n]$ is Dirichlet distributed with parameters $(1, \dots, 1) \in \Re^n$ \cite[\S6.4]{David2003}. 
So we can sample it by either sampling from a Dirichlet distribution, or by sampling $n-1$ uniform order statistics $U_{(1)} \leq U_{(2)} \leq \dots \leq U_{(n-1)}$
and using $(U_{(1)}, U_{(2)} - U_{(1)}, \dots, 1 - U_{(n-1)})$ as a sample from $\mathrm{U}[\Delta^n]$ (cf. \cite[\S6.4]{David2003}). 
We can then evaluate $I^\diamond_{\phi}(\nu)$ by solving \cref{eq:div-adj}, which is a special case of the 
optimization problem in \cite{Best2000}. That paper presents the \emph{pool-adjacent violator (PAV)} algorithm.
It solves \cref{eq:div-adj} with complexity $\mathcal{O}(n)$. 

The fact that we can sample $I^\diamond_{\phi}(\nu)$ efficiently implies that \cref{eq:calibration-scalar} can be estimated
through data-driven means.
\begin{theorem} \label{thm:divergence-bound}
    Let $\alpha_1, \dots, \alpha_m$ denote $m$ iid samples from $I^\diamond_{\phi}(\nu)$ with $\nu \deq \mathrm{U}[\Delta^n]$ and 
    $\alpha_{(1)} \leq \dots \leq \alpha_{(m)}$ the associated order statistics. Then, for any $\delta \in [0, 1]$ 
    and $k \in [m]$, 
    \begin{equation} \label{eq:confidence-upper-bound}
        \prob\left[ \alpha_\star \leq \alpha_{(k)} \right] \geq 1 - \beta
    \end{equation}
    with $\alpha_\star$ as in \cref{prop:divergence-bound-pre}, and $\beta = I_{1-\delta}(k, m-k+1)$ 
    the regularized incomplete beta function (i.e., the cdf of a beta distribution) at level $1-\delta$.
\end{theorem}
\begin{proof}
    The result follows directly from \cref{prop:divergence-bound-pre} and a one-sided data-driven bound of a quantile 
    stated in \cite[\S{}G.2.2]{Meeker2017} applied to \cref{eq:quantile-function}.
    Alternatively \cref{eq:quantile-function} can be interpreted as a \emph{scenario program}, 
    for which \cite[Thm.~3.7]{Campi2018} holds.
\end{proof}

In practice, the user would select $m$ (the number of samples from $I^\diamond_{\phi}(\nu)$ computed using PAV).
A larger value gives a tighter upper bound for $\alpha_\star$ at the cost of additional computation time. The confidence level is then determined 
by fixing some $\beta \in [0, 1]$ and then finding the smallest $k$ such that $I_{1-\delta}(k, m-k+1) \leq \beta$. 
Such a $k$ is determined with a scalar root finder\footnote{We use \texttt{brentq} as implemented in \texttt{scipy 1.10.0}\label{foot:brentq}}.

We can establish a connection with other results in literature, through the following corollary of \cref{cor:mean-bound-main}:
\begin{corollary} \label{cor:simple-distortion}
    Let $\mu \in \Re^n_{\uparrow} \cap \Delta^n$. Then, for uniform order statistics $U_{(1)} \leq U_{(2)} \leq \dots \leq U_{(n-1)} \leq U_{(n)} = 1$
    and $Z_{(1)} \leq Z_{(2)} \leq \dots Z_{(n-1)}$ and $Z_{(n)} = \esssup[Z]$ the order statistics of rv $Z \colon \Omega \to \Re$,
    \begin{equation} \label{eq:distortion-confidence}
        \prob\left[ \sum_{i=1}^{n} \mu_i Z_{(i)} \geq \E[Z] \right] \geq \prob\left[ \sum_{i=1}^{k} \mu_i \leq U_{(k)}, \, \forall k \in [n] \right].
    \end{equation}
\end{corollary}
\begin{proof}
    Using \cite{Bertsimas2009b} gives $\sum_{i=1}^{n} \mu_i Z_{(i)} = \sup_{\mu \in \amb} \, \sum_{i=1}^{n} \mu_i Z_{i}$,
    with $\amb = \hull{\Pi^n \mu}$ the convex hull of all permutations of $\mu$.
    Invoking \cref{cor:mean-bound-main} then gives $\prob\left[ \sum_{i=1}^{n} \mu_i Z_{(i)} \geq \E[Z] \right] \geq \prob[\nu \in \amb + \major{n}]$,
    with $\nu$ distributed according to $\mathrm{U}[\Delta^n]$. 
    Assume that  $\sum_{i=1}^{k} \mu_i \leq U_{(k)}$ for all $k \in [n-1]$.
    Noting that $\sum_{i=1}^{k} \nu_i \deq U_{(k)}$ by \cite[\S6.4]{David2003}
    and using \cref{eq:majorization-cone}, this implies that 
    $\nu - \mu \in \major{n}$ and, since $\mu \in \amb$, $\nu \in \amb + \major{n}$. 
    So $\prob[\nu \in \amb + \major{n}] \geq \prob[\sum_{i=1}^{k} \mu_i \leq U_{(k)}, \, \forall k \in [n]]$.
\end{proof}
 
Evaluating the right-hand side of \cref{eq:distortion-confidence} was studied in the context of 
lower bounding the cdf of a random variable (cf.\ \cite{Moscovich2020} for an efficient and numerically stable algorithm). 

The expression $\sum_{i=1}^{n} \mu_i Z_{(i)}$ is called a \emph{distortion risk} \cite{Bertsimas2009b},
which is a convex function of $(Z_1, \dots, Z_n)$ iff $\mu \in \Re^n_{\uparrow}$. Since the uniform order statistics form 
a uniform distribution over a convex set (i.e., $\Re^n_{\uparrow} \cap [0, 1]^n$), their density is quasi-concave \cite[Ex.~4.10]{Shapiro2021}. 
Hence we can leverage \cite[Ex.~4.17, Cor.~4.42]{Shapiro2021} to claim that the calibration problem
\begin{equation*}
    \inf_{\mu \in \Re^n_{\uparrow} \cap \Delta^n} \, \left\{\varphi(\mu) \colon \prob\left[\sum_{i=1}^k \mu_i \leq U_{(k)}, \, \forall k \in [n]\right] \geq 1 - \delta \right\}
\end{equation*}
is a convex problem, whenever $\varphi$ is convex. We will study the choice of $\varphi$, 
which should measure the size of $\amb$, in future work as well as the tractability of the calibration problem. 
For now, we instead consider a specific parametrization:
\begin{equation} \label{eq:cvar-weights}
    \mu^{(\gamma)}\dfn \left(0, \dots, 0, \tfrac{\lceil (n-1)\gamma \rceil}{n-1} - \gamma, \tfrac{1}{n-1}, \dots, \tfrac{1}{n-1}, \gamma\right),
\end{equation}
with $\gamma \geq 1/(n-1)$, such that $\mu \in \Delta^n \cap \Re_\uparrow^n$ holds. Then 
\begin{equation} \label{eq:biased-cvar}
    \sum_{i=1}^{n} \mu_i^{(\gamma)} Z_{(i)} =  \left( \tfrac{d}{n-1} - \gamma \right) Z_{(d)} + \sum_{i=d+1}^{n-1} \tfrac{Z_{(i)}}{n-1} + \gamma Z_{(n)},
\end{equation}
with $d = \lceil (n-1)\gamma\rceil$. The final expression is a well-known bound for the expectation of $Z$ due to Anderson \cite{Anderson1969}.
The re-interpretation in terms of a distortion risk is novel to the authors' knowledge. In \ilpub{the technical report \cite[App.~A-B]{TR}}\ilarxiv{\cref{app:cvar}}
we also show that the final expression is an affine transformation of the well-known conditional value-at-risk. Hence we denote it as $\bar{\CVAR}$ 
in the experiments. The value of $\gamma$ in the context of \cite{Anderson1969} is usually determined using an 
asymptotic bound (cf.\ \cite[Thm. 11.6.2]{Wilks1964}). We find an accurate value by solving:
\begin{equation} \label{eq:calibration-cvar}
    \inf_{\gamma \geq 1/(n-1)} \, \left\{\gamma \colon \prob\left[ \sum_{i=1}^{k} \mu_i^{(\gamma)} \leq U_{(k)}, \, \forall k \in [n] \right] \geq 1- \delta\right\},
\end{equation}
for a user-specified $\delta \in [0, 1]$ with a scalar root finder\footref{foot:brentq}.
The probability is evaluated numerically using \cite{Moscovich2020}. The calibration problem \cref{eq:calibration-cvar} in \cite{Anderson1969}
is interpreted as moving the empirical cdf down as little as possible, while still guaranteeing that it lower bounds the true cdf with high 
probability. It is also comparable to producing the tightest mean bound as in the calibration problem \cref{eq:calibration-scalar}.

%% file: assets/onesided/onesideddominated.tex
\begin{tikzpicture}
    \begin{axis}[
        cdf axis, 
        clip=false, 
        xtick={0.4, 0.55, 1.0}, 
        xticklabels={$Z_{(1)}$, $Z_{(2)}$, $Z_{(3)}$}, 
        width=8cm, 
        height=4cm
    ] 
        \addplot[thick, mGreen] table[x=x, y=y, col sep=comma] {assets/onesided/betainc.csv} node[above, mGreen] {$F$};
        \addplot[draw=none, fill=mGreen, fill opacity=0.1] table[x=x, y=y, col sep=comma] {assets/onesided/betainc.csv} -- (0, 1.0) -- (0.0, 0.0);
        \weightedcdf[mBlue!60!white, thin]{0.4/0.26, 0.55/0.65, 1.0/1.0}{3}
        \weightedcdf[mRed, semithick]{0.4/0.12, 0.55/0.5, 1.0/1.0}{3}
        \draw [ultra thin, dashed, mRed] (0.0, 0.12) -- (1.1, 0.12) node [right] {$\mu_1$};
        \draw [ultra thin, dashed, mRed] (0.0, 0.5) -- (1.1, 0.5) node [right] {$\mu_1 + \mu_2$};

        \draw [ultra thin, dashed, mBlue!60!white] (0.0, 0.26) -- (1.1, 0.26) node [right] {$W_1$};
        \draw [ultra thin, dashed, mBlue!60!white] (0.0, 0.65) -- (1.1, 0.65) node [right] {$W_1 + W_2$};
    \end{axis}
\end{tikzpicture}

%% file: content-old/cases.tex
\section{Case Studies} \label{sec:case-studies} 

To illustrate the validity and potential of our method we provide several simple case studies. These 
are convex for maximum interpretability, as in the non-convex case a worse generalization performance might be caused by local optima. 
Nonetheless our method is also applicable in non-convex settings, where stochastic gradient
descent methods can be used (cf. \cite{Mehta2022} for simple distortions and \cite{Chouzenoux2019} for divergences).
In the convex case we use duality to reformulate the proxy cost in \cref{eq:ordered-risk}. See \cite{Schuurmans2023,Bertsimas2009b} for details.

We present problems of the form \cref{eq:erm} and employ ordered risk minimization \cref{eq:ordered-risk} 
or using \cref{eq:biased-cvar}, which we refer to as $\bar{\CVAR}$. For divergences we use either the \emph{total variation (TV)} or
\emph{Kullback-Leibler (KL)} and the radius is calibrated using \cref{thm:divergence-bound} (with $\beta = 0.005$ and $m=10\,000$).
The value of $\gamma$ in \cref{eq:biased-cvar} is calibrated using \cref{eq:calibration-cvar}.


\subsection{Newsvendor}
We begin with a toy problem, illustrating the behavior of our method in low-sample settings. 
Let $\xi \colon \Omega \to \Re$ be Beta distributed with $\alpha = 0.1$, $\beta = 0.2$, scaled by a factor $\bar{D} \dfn 100$. 
Consider a newsvendor problem \cite[\S1.2.1]{Shapiro2021}:
\begin{equation*} \label{eq:newsvendor}
    \begin{alignedat}{2}
        &\minimize_{\theta \in \Re} & \quad & \E\left[\smashunderbracket{c\theta + b[\xi - \theta]_+ +  h[\theta - \xi]_+}{\ell(\theta, \xi)}\right],
    \end{alignedat}
\end{equation*}
with $b = 14$, $h = 2$ amd $c = 1$. For samples $\{\xi_i\}_{i=1}^{n-1}$ with $n=20$ let $\ell_i(\theta) = \ell(\theta, \xi_i)$ 
for $i \in [n-1]$, $\ell_n(\theta) = \max\{(c - b)\theta + b \bar{D}, (c + h)\theta\}$ a robust upper bound. 
We replace the expectation by a data-driven proxy as described at the start of the section.
For the \emph{sample average approach (SAA)} we take $\sum_{i=1}^{n-1} \ell_i(\theta)/(n-1)$.

\begin{figure}
    \centering
    \vspace{0.2em}
    \input{assets/boxes.tex}\vspace{-0.5em}
    \caption{Box plots showing newsvendor expected cost (left); and difference between the predicted cost and expected cost (right). The colored area is the \emph{inter-quartile range (IQR)}, while the whiskers show the 
    range of samples truncated to $1.5$ times the IQR. Outliers outside of this range are depicted as diamonds. The red dashed lines depict the robust performance. The blue dashed line is the optimal cost.} \label{fig:newsvendor-example}%
    \vspace{-1em}
\end{figure}
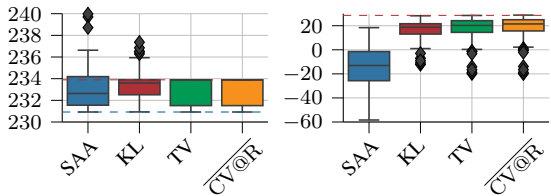

The calibration problems are solved for $\delta = 0.2$. 
Their performance is compared over $200$ sampled data sets in \cref{fig:newsvendor-example}.
The left plot shows the actual expected cost for the minimizers.
The blue dashed line is the true optimum of \cref{eq:newsvendor}. See \cite[\S1.2.1]{Shapiro2021} for details on how to compute these values. 
Note how the SAA performs decently in the median, but has significantly more variance. The outliers above $240$ were omitted, the largest of which was $428.2$. 
Moreover, the right plot depicts the difference between the optimum value of the proxy cost, and the true cost. 
The SAA often underestimates its true cost, while our methods overestimate it. The dashed red line depicts the behavior when taking $\amb = \Delta^n$
in \cref{eq:ordered-risk} (cf. \cite[Eq.~1.9]{Shapiro2021}). As we almost never perform worse than this robust method, 
this shows that our methods learn from data without over-fitting on the sample. 

In large sample cases, we can use the largest sample as an approximation of $\ell_n(\theta)$. 
This heuristic is similar to the one used in the scenario approach, the consequences of which have been studied in detail (cf.\ \cite{Ramponi2018}).
In combination with some regularization, this significantly boosts 
the performance of our method, as shown in the next \ilarxiv{examples}\ilpub{example}.

\begin{arxiv}
\subsection{Regression}
Let $T_k \colon \Re \to \Re$ denote the Chebychev polynomials of the first kind for $k \geq 0$ and $f_d(x) = (T_k(x))_{k=0}^{d} \in \Re^{d+1}$ 
a feature vector. Consider a lasso regression problem:
\begin{equation} \label{eq:regression}
    \begin{alignedat}{2}
        &\minimize_{\theta \in \Re^{d+1}} &\qquad & \E\left[(\<f_d(X), \theta\> - Y)^2\right] + \lambda \nrm{\theta}_1.
    \end{alignedat}    
\end{equation}
Assuming access to samples $\{(X_i, Y_i)\}_{i=1}^n$, we replace the expectation with the proxy costs described above,
where $\ell_i(\theta) = (\<f_d(X_i), \theta\> - Y_i)^2$ for $i \in [n]$. 
So we approximate the robust term with the largest sample.

For the parameters $\theta_\star = (0, 0, 0.2, 0.5, 1.0)$ the data is generated as $Y_i = \<f_4(X_i), \theta_\star\> + E_i$ with $X_i \deq \mathcal{U}(-1, 1)$ 
and $E_i \deq \mathcal{U}(-0.2, 0.2)$ for $i \in [n]$. We over-parametrize the problem, taking $d = 20$, to illustrate the regularizing effect of our method.
A fit is plotted for $\lambda = 0.2$ and $n = 50$ in \cref{fig:regression-example}. Note how the risk measures all perform similarly, 
while SAA has a worse fit.

\begin{figure}
    \centering
    \input{assets/regression.tex}
    \caption{Regression using $n = 50$ samples with $d=20$ and $\lambda = 0.2$ for different risk measures.} \label{fig:regression-example}
    \vspace{-1em}
\end{figure}
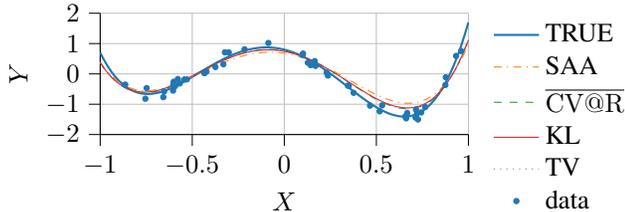

The methods are evaluated quantitatively by sampling an additional $100\,000$ data points and computing a sample approximation of the cost of \cref{eq:regression}.
The resulting performance is compared for several tunings in \cref{tab:regression-quantitative}, where any parameters not mentioned are kept as specified above. 
It is of note that our methods are significantly less sensitive to tuning parameters compared to the SAA. In fact, our methods outperform SAA for all tunings investigated.

{\setlength{\tabcolsep}{4pt}\tiny
\begin{table}
    \vspace{1.2em}
   \input{assets/regression-table.tex}
   \centering
   \caption{Regression generalization performance for various tuning parameters. Values are reported as \emph{mean (standard deviation $\cdot\, 10^{3}$)}
   computed over $10$ training sets. The same $10$ sets were used for every selection of parameters and method. Note that $\varepsilon$ 
   does not affect SAA.} \label{tab:regression-quantitative} \vspace{-1em}
\end{table}
}
\end{arxiv}

\subsection{Support Vector Machines}
Consider a classification problem with $X \deq \mathcal{N}(0, I_2)$ normally distributed and $Y = 1$ if $X_{1} X_{2} \geq 0$ 
and $Y = -1$ otherwise
. A \emph{Support Vector Machine (SVM)}
solves:
\begin{equation*}
    \begin{alignedat}{2}
        &\minimize_{(f, b) \in \set{H} \times \Re} &\qquad & \frac{1}{2} \nrm{f}_{\set{H}}^2 + \lambda\E\left[ 1 - Y (f(X) - b) \right]_+
    \end{alignedat}
\end{equation*}
with $\lambda > 0$ and $\set{H}$ some \emph{reproducing kernel Hilbert Space (RKHS)} \cite[Def.~2.9]{Scholkopf2002}. The 
resulting classifier is then given by $\mathrm{sign}(f(X) - b)$. Henceforth $\set{H}$ is the RKHS associated with the \emph{radial basis function} kernel 
\cite[\S2.3]{Scholkopf2002} with some standard deviation $\sigma$.
Solving the primal problem is difficult for two reasons: \emph{(i)} 
the true expectation is often unknown; \emph{(ii)} optimizing over the infinite dimensional $\set{H}$ is intractable in general. 
We resolve \emph{(i)} by replacing the expectation with a proxy-cost as described above and \emph{(ii)} through the usual duality trick \cite[\S7.4]{Scholkopf2002}.
Details are deferred to \ilarxiv{\cref{app:svm}}\ilpub{\cite[App.~B]{TR}}. 

The proxy cost of three of the risks above -- SAA, TV and $\bar{\CVAR}$ -- is a maximum of linear functions and the dual problem is a QP. 
The sample average -- C-SVC in \cite[\S7.5]{Scholkopf2002} -- is the usual choice. We illustrate the superior performance our calibrated risks. 

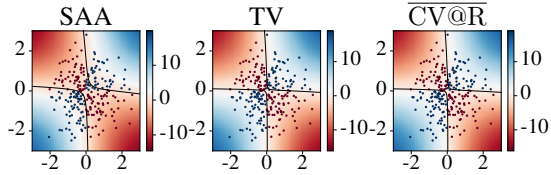
\begin{figure}
    \vspace{1em}
    \centering
    \def\svgwidth{0.85\columnwidth}
    \input{assets/svm.tex}
    \caption{SVM classifiers trained using $n = 250$ samples with $\sigma=0.25$ and $\lambda = 10^4$ for different risks. The red and blue markers are samples for $Y=1$ and $-1$ respectively. 
    The line is the decision boundary and the color axis depicts $f(X) - b$.} \label{fig:svm-example}
    \vspace{-0.5em}
\end{figure}

In \cref{fig:svm-example}, the three classifiers produced by the three proxy costs above are depicted. 
Note how both TV and $\bar{\CVAR}$ perform similarly and both visibly better than the usual SAA. 
Quantitative performance is compared through the fraction of 
incorrectly labeled samples in a test set of $10^5$ samples, which 
we refer to as the misclassification rate. The performance is compared for several tunings in \cref{tab:svm-quantitative}, 
where any parameters not mentioned are kept as specified above. It is of note that our methods are significantly less 
sensitive to tuning parameters compared to the SAA. In fact, even for the tunings where SAA performs best, our methods 
perform better for the same tuning, for reasonable choices of $\delta$. 

{\setlength{\tabcolsep}{4pt}
\begin{table}
    \vspace{1.2em}
    \centering
\input{assets/svm-table.tex}
    \caption{%
    SVM misclassification rates for various tuning parameters. Reported values are the \emph{mean (standard deviation)}
    over $10$ training sets. The same $10$ sets were used for every parameter selection and method. Note that $\delta$ 
    does not affect SAA. The lowest values in a column are bold. Observe that in the rows where SAA achieves its best performance, our methods still perform better.%
    } \label{tab:svm-quantitative} \vspace{-1.5em}
\end{table}
}

\begin{figure}
    \centering
    \input{assets/svm-complexity.tex}\vspace{-0.5em}
    \caption{SVM misclassification rates for varying sample counts $n$. 
    The center line depicts the mean, while the intervals depicts the empirical $0.2$-confidence interval.} \label{fig:svm-complexity} \vspace{-1.5em}
\end{figure}
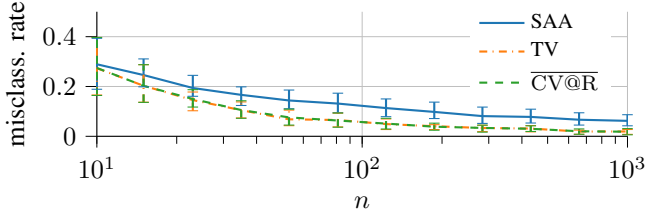

We can also examine the effect of varying the sample count $n$. For each such value we train the classifiers, again using the parameters used to produce \cref{fig:svm-example}, 
for $30$ training sets. The resulting misclassification rates are depicted in \cref{fig:svm-complexity}. Again note that $\bar{\CVAR}$ and TV both outperform SAA.

%% file: assets/boxes.tex
\begin{tikzpicture}

    \definecolor{brown1926061}{RGB}{192,60,61}
    \definecolor{darkgray176}{RGB}{176,176,176}
    \definecolor{darkslategray61}{RGB}{61,61,61}
    \definecolor{peru22412844}{RGB}{224,128,44}
    \definecolor{seagreen5814558}{RGB}{58,145,58}
    \definecolor{steelblue49115161}{RGB}{49,115,161}
    
    \colorlet{brown1926061}{BurntOrange}
    \colorlet{peru22412844}{Maroon}
    \colorlet{seagreen5814558}{ForestGreen}
    
    \begin{groupplot}[
        group style={group size=2 by 1},
        numeric axis,
        width=0.5\columnwidth,
        height=0.35\columnwidth,
        xtick={0,1,2,3},
        xticklabels={SAA,KL,TV,$\bar{\CVAR}$},
        xmin=-0.5, xmax=3.5,
        xticklabel style={rotate=45, yshift=0.5em},
        tick label style={font=\footnotesize},
    ]
    \nextgroupplot[
        ymin=230, ymax=240,
    ]
    \path [draw=darkslategray61, fill=steelblue49115161, semithick]
    (axis cs:-0.4,231.550197120447)
    --(axis cs:0.4,231.550197120447)
    --(axis cs:0.4,234.191226896709)
    --(axis cs:-0.4,234.191226896709)
    --(axis cs:-0.4,231.550197120447)
    --cycle;
    \path [draw=darkslategray61, fill=peru22412844, semithick]
    (axis cs:0.6,232.514582746267)
    --(axis cs:1.4,232.514582746267)
    --(axis cs:1.4,233.895616629797)
    --(axis cs:0.6,233.895616629797)
    --(axis cs:0.6,232.514582746267)
    --cycle;
    \path [draw=darkslategray61, fill=seagreen5814558, semithick]
    (axis cs:1.6,231.513291733636)
    --(axis cs:2.4,231.513291733636)
    --(axis cs:2.4,233.89561610775)
    --(axis cs:1.6,233.89561610775)
    --(axis cs:1.6,231.513291733636)
    --cycle;
    \path [draw=darkslategray61, fill=brown1926061, semithick]
    (axis cs:2.6,231.513291730288)
    --(axis cs:3.4,231.513291730288)
    --(axis cs:3.4,233.89561610782)
    --(axis cs:2.6,233.89561610782)
    --(axis cs:2.6,231.513291730288)
    --cycle;
    \addplot [semithick, darkslategray61]
    table {%
    0 231.550197120447
    0 230.922653028514
    };
    \addplot [semithick, darkslategray61]
    table {%
    0 234.191226896709
    0 236.628514347995
    };
    \addplot [semithick, darkslategray61]
    table {%
    -0.2 230.922653028514
    0.2 230.922653028514
    };
    \addplot [semithick, darkslategray61]
    table {%
    -0.2 236.628514347995
    0.2 236.628514347995
    };
    \addplot [black, mark=diamond*, mark size=2.5, mark options={solid,fill=darkslategray61}, only marks]
    table {%
    0 309.27480996047
    0 247.239662186754
    0 272.054556457494
    0 241.956296049462
    0 251.597605495105
    0 247.192043511445
    0 244.233368416658
    0 252.86600364674
    0 294.508720637123
    0 239.699528761407
    0 282.95835734537
    0 240.866292335017
    0 260.153936686469
    0 241.159137730683
    0 254.503548424132
    0 349.428745382197
    0 238.706990281232
    0 342.008656823814
    0 241.989558789093
    0 375.043430031865
    0 279.95704605889
    0 285.303883946213
    0 241.311335337042
    0 256.120889572637
    0 332.920304106216
    0 283.860290535548
    0 248.038686798161
    0 428.193625997454
    0 350.00539634576
    0 257.775349963224
    0 279.026151784935
    0 286.705272187652
    0 320.624346361995
    0 269.888760341798
    0 310.741096286397
    0 265.717927708245
    0 240.576552136251
    0 241.186957057808
    0 239.976526233234
    0 311.911146978049
    0 240.879249166674
    0 252.233515335486
    };
    \addplot [semithick, darkslategray61]
    table {%
    1 232.514582746267
    1 230.930341261761
    };
    \addplot [semithick, darkslategray61]
    table {%
    1 233.895616629797
    1 235.934460789289
    };
    \addplot [semithick, darkslategray61]
    table {%
    0.8 230.930341261761
    1.2 230.930341261761
    };
    \addplot [semithick, darkslategray61]
    table {%
    0.8 235.934460789289
    1.2 235.934460789289
    };
    \addplot [black, mark=diamond*, mark size=2.5, mark options={solid,fill=darkslategray61}, only marks]
    table {%
    1 236.468204369368
    1 236.282442484147
    1 236.5232542056
    1 241.019801966879
    1 236.844590076523
    1 237.365715769178
    };
    \addplot [semithick, darkslategray61]
    table {%
    2 231.513291733636
    2 230.922503076188
    };
    \addplot [semithick, darkslategray61]
    table {%
    2 233.89561610775
    2 233.895616125469
    };
    \addplot [semithick, darkslategray61]
    table {%
    1.8 230.922503076188
    2.2 230.922503076188
    };
    \addplot [semithick, darkslategray61]
    table {%
    1.8 233.895616125469
    2.2 233.895616125469
    };
    \addplot [semithick, darkslategray61]
    table {%
    3 231.513291730288
    3 230.922503076182
    };
    \addplot [semithick, darkslategray61]
    table {%
    3 233.89561610782
    3 233.895616121114
    };
    \addplot [semithick, darkslategray61]
    table {%
    2.8 230.922503076182
    3.2 230.922503076182
    };
    \addplot [semithick, darkslategray61]
    table {%
    2.8 233.895616121114
    3.2 233.895616121114
    };
    \addplot [thin, dashed, MidnightBlue]
    table {%
    -0.5 230.922482721782
    3.5 230.922482721782
    };
    \addplot [thin, dashed, Maroon]
    table {%
    -0.5 233.895616107733
    3.5 233.895616107733
    };
    \addplot [semithick, darkslategray61]
    table {%
    -0.4 232.632500440114
    0.4 232.632500440114
    };
    \addplot [semithick, darkslategray61]
    table {%
    0.6 233.596504457635
    1.4 233.596504457635
    };
    \addplot [semithick, darkslategray61]
    table {%
    1.6 233.884876617215
    2.4 233.884876617215
    };
    \addplot [semithick, darkslategray61]
    table {%
    2.6 233.884876618234
    3.4 233.884876618234
    };
    
    \nextgroupplot[
        ymin=-60, ymax=30,
    ]
    \path [draw=darkslategray61, fill=steelblue49115161, semithick]
    (axis cs:-0.4,-25.6890308114622)
    --(axis cs:0.4,-25.6890308114622)
    --(axis cs:0.4,-1.45048896012033)
    --(axis cs:-0.4,-1.45048896012033)
    --(axis cs:-0.4,-25.6890308114622)
    --cycle;
    \path [draw=darkslategray61, fill=peru22412844, semithick]
    (axis cs:0.6,13.0663726881325)
    --(axis cs:1.4,13.0663726881325)
    --(axis cs:1.4,21.6590523033854)
    --(axis cs:0.6,21.6590523033854)
    --(axis cs:0.6,13.0663726881325)
    --cycle;
    \path [draw=darkslategray61, fill=seagreen5814558, semithick]
    (axis cs:1.6,14.4910449726339)
    --(axis cs:2.4,14.4910449726339)
    --(axis cs:2.4,24.2024290131935)
    --(axis cs:1.6,24.2024290131935)
    --(axis cs:1.6,14.4910449726339)
    --cycle;
    \path [draw=darkslategray61, fill=brown1926061, semithick]
    (axis cs:2.6,15.6543825861303)
    --(axis cs:3.4,15.6543825861303)
    --(axis cs:3.4,24.9936781479305)
    --(axis cs:2.6,24.9936781479305)
    --(axis cs:2.6,15.6543825861303)
    --cycle;
    \addplot [semithick, darkslategray61]
    table {%
    0 -25.6890308114622
    0 -58.347480844538
    };
    \addplot [semithick, darkslategray61]
    table {%
    0 -1.45048896012033
    0 18.4005333648362
    };
    \addplot [semithick, darkslategray61]
    table {%
    -0.2 -58.347480844538
    0.2 -58.347480844538
    };
    \addplot [semithick, darkslategray61]
    table {%
    -0.2 18.4005333648362
    0.2 18.4005333648362
    };
    \addplot [black, mark=diamond*, mark size=2.5, mark options={solid,fill=darkslategray61}, only marks]
    table {%
    0 -102.757610901079
    0 -71.3284813868621
    0 -76.1319872188375
    0 -196.488409076203
    0 -148.785854928501
    0 -174.89721207012
    0 -73.6855960407518
    0 -127.680676068894
    0 -81.3313453018287
    0 -249.243126843006
    0 -170.1697672745
    0 -66.2772007807008
    0 -99.9473506213679
    0 -166.24896593697
    0 -72.7900241581331
    0 -111.037617215037
    0 -111.115481470117
    };
    \addplot [semithick, darkslategray61]
    table {%
    1 13.0663726881325
    1 1.00876390880111
    };
    \addplot [semithick, darkslategray61]
    table {%
    1 21.6590523033854
    1 28.2427710251418
    };
    \addplot [semithick, darkslategray61]
    table {%
    0.8 1.00876390880111
    1.2 1.00876390880111
    };
    \addplot [semithick, darkslategray61]
    table {%
    0.8 28.2427710251418
    1.2 28.2427710251418
    };
    \addplot [black, mark=diamond*, mark size=2.5, mark options={solid,fill=darkslategray61}, only marks]
    table {%
    1 -12.2298209716486
    1 -7.02612768771002
    1 -10.914276710663
    1 -2.96599017393112
    1 -9.01930981768081
    };
    \addplot [semithick, darkslategray61]
    table {%
    2 14.4910449726339
    2 0.337087428241631
    };
    \addplot [semithick, darkslategray61]
    table {%
    2 24.2024290131935
    2 28.5215051020295
    };
    \addplot [semithick, darkslategray61]
    table {%
    1.8 0.337087428241631
    2.2 0.337087428241631
    };
    \addplot [semithick, darkslategray61]
    table {%
    1.8 28.5215051020295
    2.2 28.5215051020295
    };
    \addplot [black, mark=diamond*, mark size=2.5, mark options={solid,fill=darkslategray61}, only marks]
    table {%
    2 -1.22791451928057
    2 -19.8360947817637
    2 -1.30310912532676
    2 -14.0482167319997
    2 -18.8023431435918
    2 -3.21427921188493
    2 -16.6410930116488
    };
    \addplot [semithick, darkslategray61]
    table {%
    3 15.6543825861303
    3 2.20596381880588
    };
    \addplot [semithick, darkslategray61]
    table {%
    3 24.9936781479305
    3 28.819354160853
    };
    \addplot [semithick, darkslategray61]
    table {%
    2.8 2.20596381880588
    3.2 2.20596381880588
    };
    \addplot [semithick, darkslategray61]
    table {%
    2.8 28.819354160853
    3.2 28.819354160853
    };
    \addplot [black, mark=diamond*, mark size=2.5, mark options={solid,fill=darkslategray61}, only marks]
    table {%
    3 -0.467573940775821
    3 -19.9357523971958
    3 1.24172475214672
    3 -0.328819714471848
    3 -13.8513716980756
    3 -18.7944169895371
    3 -2.17430874576158
    3 -16.5346733609058
    };
    \addplot [thin, dashed, Maroon]
    table {%
    -0.5 28.6043838922666
    3.5 28.6043838922666
    };
    \addplot [semithick, darkslategray61]
    table {%
    -0.4 -13.0558730299372
    0.4 -13.0558730299372
    };
    \addplot [semithick, darkslategray61]
    table {%
    0.6 18.6996572945657
    1.4 18.6996572945657
    };
    \addplot [semithick, darkslategray61]
    table {%
    1.6 20.2181499801056
    2.4 20.2181499801056
    };
    \addplot [semithick, darkslategray61]
    table {%
    2.6 21.4778487560043
    3.4 21.4778487560043
    };
    \end{groupplot}
    
    \end{tikzpicture}
    

%% file: assets/regression.tex
\begin{tikzpicture}

    \definecolor{crimson_}{RGB}{214,39,40}
    \definecolor{darkgray_}{RGB}{176,176,176}
    \definecolor{darkorange_}{RGB}{255,127,14}
    \definecolor{forestgreen_}{RGB}{44,160,44}
    \definecolor{mediumpurple_}{RGB}{148,103,189}
    \definecolor{steelblue_}{RGB}{31,119,180}
    
    \begin{axis}[
        numeric axis,
        width=0.75\columnwidth,
        height=0.37\columnwidth,
        xlabel={$X$},
        ylabel={$Y$},
        xmin=-1, xmax=1,
        ymin=-2, ymax=2,
        crimson/.style={crimson_},
        darkgray/.style={darkgray_, densely dashed},
        darkorange/.style={darkorange_, dash dot},
        forestgreen/.style={forestgreen_, dashed},
        mediumpurple/.style={mediumpurple_, dotted},
        steelblue/.style={steelblue_, thick},
        legend style={at={(axis cs: 1, 2)}, anchor=north west, draw=none, yshift=0pt, xshift=6pt},
        legend cell align={left},
    ]
        \addlegendimage{steelblue}
        \addlegendentry{TRUE};
        \addlegendimage{darkorange}
        \addlegendentry{SAA};
        \addlegendimage{forestgreen}
        \addlegendentry{$\bar{\CVAR}$};
        \addlegendimage{crimson}
        \addlegendentry{KL};
        \addlegendimage{mediumpurple}
        \addlegendentry{TV};
        \addlegendimage{draw=steelblue_, fill=steelblue_, mark=*, only marks, mark options={mark size=1pt, line width=0pt}}
        \addlegendentry{data};
    \addplot [draw=steelblue_, fill=steelblue_, mark=*, only marks, mark options={mark size=1pt, line width=0pt}]
    table{%
    x  y
    -0.543166182581587 -0.186439021733187
    -0.432104771489604 0.0308442977855478
    0.726162628497975 -1.50338914633486
    -0.434105075042275 0.0261961360724023
    0.715094085483988 -1.17683624265015
    0.16481612752855 0.421962097245791
    -0.086395073130481 1.0191517877377
    0.960604840561078 0.753690245809822
    -0.215895483786743 0.810975608444219
    0.656624362854402 -1.4061217161301
    -0.754778321812374 -0.817913533887828
    0.765150534195686 -1.0932957403547
    0.143022866884548 0.288495071307838
    -0.603235829669007 -0.360577334406098
    0.531203462364422 -1.03792982238258
    -0.747694763211481 -0.468299561899551
    -0.578216529780555 -0.168805911381481
    -0.652657591813641 -0.579135702603773
    -0.37395195156126 0.227134901888414
    0.465578114554053 -1.04649787822397
    -0.422555017728667 0.081403975075059
    -0.301567597977295 0.707169506207746
    -0.592960845848694 -0.403535120168605
    0.712889170958056 -1.44011748733879
    0.743720587981869 -1.27107794090422
    0.350131019657781 -0.401010268721278
    0.876957366972021 -0.112637203415097
    0.169588247914107 0.263865013196714
    0.517954974565785 -1.23468037322038
    0.661681581309915 -1.46633077530077
    0.127433231504519 0.370749741864243
    0.101719832355051 0.669341822209838
    0.376539022813597 -0.624808799718859
    0.9327774029309 0.59188392999949
    -0.32274959375947 0.711784654749942
    0.667356087426788 -1.29124845882205
    0.237059280706068 -0.0562099439344378
    0.229767777199406 0.0437313969142713
    -0.56802203298394 -0.315052239095855
    -0.330805131508898 0.317501535506178
    -0.434311621473719 0.0533249617422682
    -0.604083632534133 -0.550460503807635
    -0.657112551558907 -0.76777675653173
    -0.600567171770596 -0.258698957920366
    0.875047644076749 -0.365424160728532
    0.098605554312017 0.624870432297143
    0.345644360173752 -0.385198169024062
    -0.529212343562386 -0.182449030461903
    -0.863956353569865 -0.352402360987545
    0.720755567479581 -1.30723888296557
    };
    \addplot [steelblue]
    table {%
    -1 0.7
    -0.97979797979798 0.465308521321049
    -0.95959595959596 0.257230386580514
    -0.939393939393939 0.0743294873773215
    -0.919191919191919 -0.0847983045442793
    -0.898989898989899 -0.221525137294709
    -0.878787878787879 -0.337191178839063
    -0.858585858585859 -0.433104616997108
    -0.838383838383838 -0.510541659443287
    -0.818181818181818 -0.570746533706714
    -0.797979797979798 -0.614931487171178
    -0.777777777777778 -0.644276787075141
    -0.757575757575758 -0.659930720511738
    -0.737373737373737 -0.663009594428775
    -0.717171717171717 -0.654597735628737
    -0.696969696969697 -0.635747490768778
    -0.676767676767677 -0.607479226360726
    -0.656565656565657 -0.570781328771082
    -0.636363636363636 -0.526610204221023
    -0.616161616161616 -0.475890278786396
    -0.595959595959596 -0.419513998397724
    -0.575757575757576 -0.358341828840201
    -0.555555555555556 -0.293202255753696
    -0.535353535353535 -0.224891784632751
    -0.515151515151515 -0.154174940826581
    -0.494949494949495 -0.0817842695390749
    -0.474747474747475 -0.0084203358287942
    -0.454545454545454 0.0652482753910254
    -0.434343434343434 0.138584959352476
    -0.414141414141414 0.210985091432974
    -0.393939393939394 0.281876027155266
    -0.373737373737374 0.350717102187422
    -0.353535353535353 0.416999632342841
    -0.333333333333333 0.480246913580247
    -0.313131313131313 0.540014222003691
    -0.292929292929293 0.595888813862552
    -0.272727272727273 0.647489925551533
    -0.252525252525252 0.694468773610667
    -0.232323232323232 0.736508554725311
    -0.212121212121212 0.773324445726149
    -0.191919191919192 0.804663603589193
    -0.171717171717172 0.83030516543578
    -0.151515151515151 0.850060248532575
    -0.131313131313131 0.86377195029157
    -0.111111111111111 0.871315348270081
    -0.0909090909090908 0.872597500170753
    -0.0707070707070706 0.867557443841558
    -0.0505050505050504 0.856166197275793
    -0.0303030303030303 0.838426758612083
    -0.0101010101010101 0.814374106134378
    0.0101010101010102 0.784075198271956
    0.0303030303030305 0.747628973599421
    0.0505050505050506 0.705166350836706
    0.0707070707070707 0.656850228849066
    0.0909090909090911 0.602875486647087
    0.111111111111111 0.543468983386679
    0.131313131313131 0.478889558369079
    0.151515151515152 0.409428031040853
    0.171717171717172 0.335407200993891
    0.191919191919192 0.25718184796541
    0.212121212121212 0.175138731837955
    0.232323232323232 0.0896965926393961
    0.252525252525253 0.00130615054293193
    0.272727272727273 -0.0895498941329154
    0.292929292929293 -0.182356860924293
    0.313131313131313 -0.276568089222024
    0.333333333333333 -0.371604938271606
    0.353535353535354 -0.466856787173205
    0.373737373737374 -0.561681034881668
    0.393939393939394 -0.655403100206506
    0.414141414141414 -0.747316421811913
    0.434343434343434 -0.836682458216749
    0.454545454545455 -0.92273068779455
    0.474747474747475 -1.00465860877353
    0.494949494949495 -1.08163173923656
    0.515151515151515 -1.15278361712121
    0.535353535353535 -1.2172158002197
    0.555555555555556 -1.27399786617894
    0.575757575757576 -1.3221674125005
    0.595959595959596 -1.36073005654063
    0.616161616161616 -1.38865943551025
    0.636363636363636 -1.40489720647497
    0.656565656565657 -1.40835304635504
    0.676767676767677 -1.39790465192542
    0.696969696969697 -1.37239773981572
    0.717171717171717 -1.33064604651023
    0.737373737373737 -1.27143132834791
    0.757575757575758 -1.19350336152239
    0.777777777777778 -1.095579942082
    0.797979797979798 -0.976346885929703
    0.818181818181818 -0.834458028823167
    0.838383838383838 -0.668535226374717
    0.858585858585859 -0.477168354051354
    0.878787878787879 -0.25891530717476
    0.898989898989899 -0.0123020009212788
    0.919191919191919 0.26417762967806
    0.939393939393939 0.572061629737563
    0.95959595959596 0.912920024516866
    0.97979797979798 1.28835481942092
    1 1.7
    };
    \addplot [darkorange]
    table {%
    -1 0.383398179887586
    -0.97979797979798 0.217074146027678
    -0.95959595959596 0.0700484566224887
    -0.939393939393939 -0.0587138006362166
    -0.919191919191919 -0.170224751274808
    -0.898989898989899 -0.265473736458705
    -0.878787878787879 -0.345427313631317
    -0.858585858585859 -0.411029256457986
    -0.838383838383838 -0.463200554603054
    -0.818181818181818 -0.502839413554184
    -0.797979797979798 -0.53082125454698
    -0.777777777777778 -0.547998714572337
    -0.757575757575758 -0.555201646428184
    -0.737373737373737 -0.55323711878051
    -0.717171717171717 -0.542889416210877
    -0.696969696969697 -0.524920039240593
    -0.676767676767677 -0.500067704331796
    -0.656565656565657 -0.469048343871681
    -0.636363636363636 -0.432555106148336
    -0.616161616161616 -0.391258355326196
    -0.595959595959596 -0.345805671426971
    -0.575757575757576 -0.296821850319203
    -0.555555555555556 -0.244908903716978
    -0.535353535353535 -0.1906460591863
    -0.515151515151515 -0.134589760156426
    -0.494949494949495 -0.0772736659330063
    -0.474747474747475 -0.0192086517101667
    -0.454545454545454 0.0391171914206489
    -0.434343434343434 0.0972385564662737
    -0.414141414141414 0.154712920532074
    -0.393939393939394 0.211120544826237
    -0.373737373737374 0.266064474666842
    -0.353535353535353 0.319170539490177
    -0.333333333333333 0.370087352858742
    -0.313131313131313 0.418486312467596
    -0.292929292929293 0.464061600148147
    -0.272727272727273 0.506530181868975
    -0.252525252525252 0.545631807733786
    -0.232323232323232 0.58112901197704
    -0.212121212121212 0.612807112958124
    -0.191919191919192 0.64047421315508
    -0.171717171717172 0.663961199158919
    -0.151515151515151 0.683121741669379
    -0.131313131313131 0.697832295492732
    -0.111111111111111 0.707992099541888
    -0.0909090909090908 0.713523176838708
    -0.0707070707070706 0.714370334518092
    -0.0505050505050504 0.710501163833161
    -0.0303030303030303 0.70190604016072
    -0.0101010101010101 0.68859812300614
    0.0101010101010102 0.670613356006928
    0.0303030303030305 0.648010466934416
    0.0505050505050506 0.620870967693335
    0.0707070707070707 0.589299154319298
    0.0909090909090911 0.553422106974576
    0.111111111111111 0.513389689942798
    0.131313131313131 0.469374551623343
    0.151515151515152 0.421572124526345
    0.171717171717172 0.370200625269048
    0.191919191919192 0.315501054574182
    0.212121212121212 0.257737197270669
    0.232323232323232 0.19719562229665
    0.252525252525253 0.134185682704465
    0.272727272727273 0.0690395156668595
    0.292929292929293 0.00211204248348421
    0.313131313131313 -0.0662190314134019
    0.333333333333333 -0.135553216455042
    0.353535353535354 -0.205467238934918
    0.373737373737374 -0.275515041010575
    0.393939393939394 -0.345227780708464
    0.414141414141414 -0.414113831931253
    0.434343434343434 -0.481658784466494
    0.454545454545455 -0.54732544399517
    0.474747474747475 -0.610553832098438
    0.494949494949495 -0.670761186260972
    0.515151515151515 -0.727341959869725
    0.535353535353535 -0.779667822207595
    0.555555555555556 -0.827087658442489
    0.575757575757576 -0.868927569613269
    0.595959595959596 -0.904490872615147
    0.616161616161616 -0.933058100187741
    0.636363636363636 -0.953887000909223
    0.656565656565657 -0.966212539199406
    0.676767676767677 -0.969246895333215
    0.696969696969697 -0.962179465463677
    0.717171717171717 -0.944176861650733
    0.737373737373737 -0.914382911889188
    0.757575757575758 -0.871918660126886
    0.777777777777778 -0.815882366263781
    0.797979797979798 -0.745349506125103
    0.818181818181818 -0.659372771408446
    0.838383838383838 -0.556982069615526
    0.858585858585859 -0.437184523993491
    0.878787878787879 -0.298964473523944
    0.898989898989899 -0.141283473001539
    0.919191919191919 0.0369197067777947
    0.939393939393939 0.236729078773478
    0.95959595959596 0.459251439626163
    0.97979797979798 0.705616370308246
    1 0.976976238642674
    };
    \addplot [forestgreen]
    table {%
    -1 0.374299097805832
    -0.97979797979798 0.193098863801124
    -0.95959595959596 0.0333753608510691
    -0.939393939393939 -0.106036267779684
    -0.919191919191919 -0.226275027504499
    -0.898989898989899 -0.328454096197578
    -0.878787878787879 -0.413660830710654
    -0.858585858585859 -0.482956766549582
    -0.838383838383838 -0.537377615817217
    -0.818181818181818 -0.577933265526473
    -0.797979797979798 -0.605607776835561
    -0.777777777777778 -0.621359385068201
    -0.757575757575758 -0.626120500169121
    -0.737373737373737 -0.620797707266143
    -0.717171717171717 -0.606271767120371
    -0.696969696969697 -0.583397616366614
    -0.676767676767677 -0.553004367540862
    -0.656565656565657 -0.515895308949062
    -0.636363636363636 -0.472847904454144
    -0.616161616161616 -0.424613793255133
    -0.595959595959596 -0.371918789713442
    -0.575757575757576 -0.315462883256529
    -0.555555555555556 -0.255920238364955
    -0.535353535353535 -0.193939194629956
    -0.515151515151515 -0.130142266857011
    -0.494949494949495 -0.06512614518664
    -0.474747474747475 0.000538304794252235
    -0.454545454545454 0.0663060419707421
    -0.434343434343434 0.1316578496594
    -0.414141414141414 0.196100335611253
    -0.393939393939394 0.259165932052541
    -0.373737373737374 0.320412895751851
    -0.353535353535353 0.379425308099267
    -0.333333333333333 0.435813075182878
    -0.313131313131313 0.489211927850107
    -0.292929292929293 0.539283421745098
    -0.272727272727273 0.585714937318085
    -0.252525252525252 0.628219679807466
    -0.232323232323232 0.666536679199462
    -0.212121212121212 0.700430790173333
    -0.191919191919192 0.729692692041705
    -0.171717171717172 0.754138888695644
    -0.151515151515151 0.773611708562786
    -0.131313131313131 0.787979304584394
    -0.111111111111111 0.797135654214062
    -0.0909090909090908 0.801000559437447
    -0.0707070707070706 0.799519646809307
    -0.0505050505050504 0.79266436750168
    -0.0303030303030303 0.780431997355582
    -0.0101010101010101 0.76284563692824
    0.0101010101010102 0.739954211528679
    0.0303030303030305 0.711832471236289
    0.0505050505050506 0.678580990899535
    0.0707070707070707 0.640326170114911
    0.0909090909090911 0.597220233189162
    0.111111111111111 0.549441229090316
    0.131313131313131 0.497193031394803
    0.151515151515152 0.440705338238698
    0.171717171717172 0.380233672280685
    0.191919191919192 0.316059380682815
    0.212121212121212 0.248489635112663
    0.232323232323232 0.177857431767305
    0.252525252525253 0.104521591416235
    0.272727272727273 0.0288667594571697
    0.292929292929293 -0.0486965940236644
    0.313131313131313 -0.127732174196487
    0.333333333333333 -0.207777861447682
    0.353535353535354 -0.288345711362127
    0.373737373737374 -0.368921954732875
    0.393939393939394 -0.44896699759861
    0.414141414141414 -0.527915421304592
    0.434343434343434 -0.605175982578005
    0.454545454545455 -0.680131613605018
    0.474747474747475 -0.752139422094772
    0.494949494949495 -0.820530691315975
    0.515151515151515 -0.884610880094923
    0.535353535353535 -0.943659622769658
    0.555555555555556 -0.99693072910314
    0.575757575757576 -1.04365218416755
    0.595959595959596 -1.08302614822064
    0.616161616161616 -1.11422895660158
    0.636363636363636 -1.13641111967537
    0.656565656565657 -1.14869732285071
    0.676767676767677 -1.1501864266839
    0.696969696969697 -1.13995146706216
    0.717171717171717 -1.11703965543521
    0.737373737373737 -1.08047237903877
    0.757575757575758 -1.02924520103493
    0.777777777777778 -0.962327860491799
    0.797979797979798 -0.878664272148219
    0.818181818181818 -0.777172525968392
    0.838383838383838 -0.656744886587608
    0.858585858585859 -0.516247792872484
    0.878787878787879 -0.354521857929962
    0.898989898989899 -0.170381869918348
    0.919191919191919 0.0373832062057354
    0.939393939393939 0.270010226575623
    0.95959595959596 0.528761868091106
    0.97979797979798 0.81492663509092
    1 1.12981888363405
    };
    \addplot [crimson]
    table {%
    -1 0.367441864585245
    -0.97979797979798 0.19132149871054
    -0.95959595959596 0.0361450860282732
    -0.939393939393939 -0.0992252887395451
    -0.919191919191919 -0.215902272303634
    -0.898989898989899 -0.31497325574024
    -0.878787878787879 -0.397500377912792
    -0.858585858585859 -0.464520525184155
    -0.838383838383838 -0.517045330235942
    -0.818181818181818 -0.556061171135868
    -0.797979797979798 -0.582529170935533
    -0.777777777777778 -0.597385197705674
    -0.757575757575758 -0.601539864805602
    -0.737373737373737 -0.595878531201372
    -0.717171717171717 -0.581261301712469
    -0.696969696969697 -0.558523027135832
    -0.676767676767677 -0.528473304249101
    -0.656565656565657 -0.491896475726369
    -0.636363636363636 -0.449551630011274
    -0.616161616161616 -0.402172601189396
    -0.595959595959596 -0.35046796889051
    -0.575757575757576 -0.295121058236696
    -0.555555555555556 -0.236789939838526
    -0.535353535353535 -0.176107429830978
    -0.515151515151515 -0.113681089934434
    -0.494949494949495 -0.0500932275241012
    -0.474747474747475 0.0140991043072718
    -0.454545454545454 0.0783641067045849
    -0.434343434343434 0.142195235029488
    -0.414141414141414 0.205111198864161
    -0.393939393939394 0.266655962036069
    -0.373737373737374 0.326398742656825
    -0.353535353535353 0.383934013166838
    -0.333333333333333 0.438881500377398
    -0.313131313131313 0.490886185503277
    -0.292929292929293 0.539618304181189
    -0.272727272727273 0.584773346472196
    -0.252525252525252 0.626072056848895
    -0.232323232323232 0.66326043417053
    -0.212121212121212 0.69610973165085
    -0.191919191919192 0.724416456824264
    -0.171717171717172 0.748002371515772
    -0.151515151515151 0.766714491819227
    -0.131313131313131 0.780425088086974
    -0.111111111111111 0.789031684932045
    -0.0909090909090908 0.792457061242142
    -0.0707070707070706 0.790649250202904
    -0.0505050505050504 0.783581539326639
    -0.0303030303030303 0.77125247048197
    -0.0101010101010101 0.753685839919809
    0.0101010101010102 0.730930698291653
    0.0303030303030305 0.703061350657421
    0.0505050505050506 0.670177356481553
    0.0707070707070707 0.632403529617893
    0.0909090909090911 0.589889938285559
    0.111111111111111 0.542811905039363
    0.131313131313131 0.491370006739277
    0.151515151515152 0.435790074523714
    0.171717171717172 0.376323193790925
    0.191919191919192 0.313245704191798
    0.212121212121212 0.246859199635705
    0.232323232323232 0.17749052830907
    0.252525252525253 0.105491792704385
    0.272727272727273 0.0312403496555221
    0.292929292929293 -0.0448611896260633
    0.313131313131313 -0.12238495952024
    0.333333333333333 -0.200877839977869
    0.353535353535354 -0.279861456515644
    0.373737373737374 -0.358832180228372
    0.393939393939394 -0.437261127818231
    0.414141414141414 -0.514594161637535
    0.434343434343434 -0.59025188973865
    0.454545454545455 -0.663629665922649
    0.474747474747475 -0.73409758977719
    0.494949494949495 -0.801000506694745
    0.515151515151515 -0.863658007864594
    0.535353535353535 -0.921364430235982
    0.555555555555556 -0.973388856455252
    0.575757575757576 -1.01897511478565
    0.595959595959596 -1.05734177902417
    0.616161616161616 -1.0876821684336
    0.636363636363636 -1.10916434770897
    0.656565656565657 -1.12093112699432
    0.676767676767677 -1.12210006195772
    0.696969696969697 -1.1117634539196
    0.717171717171717 -1.08898835001381
    0.737373737373737 -1.05281654334404
    0.757575757575758 -1.00226457308618
    0.777777777777778 -0.936323724484975
    0.797979797979798 -0.853960028707706
    0.818181818181818 -0.754114262554411
    0.838383838383838 -0.635701948084864
    0.858585858585859 -0.497613352300298
    0.878787878787879 -0.338713487090075
    0.898989898989899 -0.157842109671367
    0.919191919191919 0.0461862763760282
    0.939393939393939 0.274582419798227
    0.95959595959596 0.528582321351509
    0.97979797979798 0.809447237605468
    1 1.11846369527124
    };
    \addplot [mediumpurple]
    table {%
    -1 0.37967880060242
    -0.97979797979798 0.196737291367453
    -0.95959595959596 0.0357485669770723
    -0.939393939393939 -0.104567551504342
    -0.919191919191919 -0.22544224661186
    -0.898989898989899 -0.328061208857748
    -0.878787878787879 -0.413567744929272
    -0.858585858585859 -0.483065530177239
    -0.838383838383838 -0.537621038163486
    -0.818181818181818 -0.578265675151122
    -0.797979797979798 -0.605997644312406
    -0.777777777777778 -0.621783562388248
    -0.757575757575758 -0.626559850095779
    -0.737373737373737 -0.621233916453434
    -0.717171717171717 -0.606685156202387
    -0.696969696969697 -0.583765778554211
    -0.676767676767677 -0.55330148454307
    -0.656565656565657 -0.516092009292975
    -0.636363636363636 -0.472911544529558
    -0.616161616161616 -0.42450905568197
    -0.595959595959596 -0.371608506946288
    -0.575757575757576 -0.314909006728416
    -0.555555555555556 -0.255084884960394
    -0.535353535353535 -0.192785712894525
    -0.515151515151515 -0.128636275127421
    -0.494949494949495 -0.0632365027910122
    -0.474747474747475 0.00283862393147737
    -0.454545454545454 0.0690391965519085
    -0.434343434343434 0.134840515587388
    -0.414141414141414 0.199743238420835
    -0.393939393939394 0.263273532356794
    -0.373737373737374 0.324983228149763
    -0.353535353535353 0.384449969889659
    -0.333333333333333 0.441277357711184
    -0.313131313131313 0.495095080350471
    -0.292929292929293 0.545559035102607
    -0.272727272727273 0.592351433236494
    -0.252525252525252 0.635180889398146
    -0.232323232323232 0.673782493978965
    -0.212121212121212 0.707917867841288
    -0.191919191919192 0.737375199179024
    -0.171717171717172 0.761969262646553
    -0.151515151515151 0.781541421214224
    -0.131313131313131 0.795959611504277
    -0.111111111111111 0.805118313627184
    -0.0909090909090908 0.80893850677593
    -0.0707070707070706 0.80736761204514
    -0.0505050505050504 0.800379424123744
    -0.0303030303030303 0.787974033664344
    -0.0101010101010101 0.770177742259917
    0.0101010101010102 0.747042972058851
    0.0303030303030305 0.718648172122529
    0.0505050505050506 0.685097723675302
    0.0707070707070707 0.646521846414458
    0.0909090909090911 0.603076508037075
    0.111111111111111 0.554943339101061
    0.131313131313131 0.502329555268742
    0.151515151515152 0.445467888882845
    0.171717171717172 0.384616531696352
    0.191919191919192 0.320059090419592
    0.212121212121212 0.252104556560306
    0.232323232323232 0.181087291815499
    0.252525252525253 0.107367030028378
    0.272727272727273 0.0313288964499311
    0.292929292929293 -0.0466165552565309
    0.313131313131313 -0.126033288157925
    0.333333333333333 -0.206459709586091
    0.353535353535354 -0.287408581691364
    0.373737373737374 -0.368366916050019
    0.393939393939394 -0.448795853796001
    0.414141414141414 -0.528130533210448
    0.434343434343434 -0.60577994719416
    0.454545454545455 -0.681126793569402
    0.474747474747475 -0.753527321709078
    0.494949494949495 -0.822311179574254
    0.515151515151515 -0.886781265855676
    0.535353535353535 -0.946213592561256
    0.555555555555556 -0.999857164069
    0.575757575757576 -1.04693387937203
    0.595959595959596 -1.08663846497757
    0.616161616161616 -1.1181384466825
    0.636363636363636 -1.14057416923187
    0.656565656565657 -1.15305887367121
    0.676767676767677 -1.1546788430278
    0.696969696969697 -1.14449362779951
    0.717171717171717 -1.12153636359583
    0.737373737373737 -1.08481419416918
    0.757575757575758 -1.03330881400411
    0.777777777777778 -0.965977145608905
    0.797979797979798 -0.881752167690864
    0.818181818181818 -0.779543911503324
    0.838383838383838 -0.658240643833175
    0.858585858585859 -0.516710256337222
    0.878787878787879 -0.35380188219086
    0.898989898989899 -0.168347762187379
    0.919191919191919 0.0408346166537981
    0.939393939393939 0.274940086540748
    0.95959595959596 0.53517304563967
    0.97979797979798 0.822744365667522
    1 1.13886794507319
    };
    \end{axis}
\end{tikzpicture}
    

%% file: assets/regression-table.tex
\begin{tabular}{lllllll}
    \toprule
    &  &  & SAA & TV & $\bar{\CVAR}$ & KL \\
    $\varepsilon$ & $d$ & $\lambda$ &  &  &  &  \\
    \midrule
    \multirow[t]{10}{*}{$0.05$} & \multirow[t]{5}{*}{$10$} & $0.001$ & $0.019\,(03)$ & $0.018\,(02)$ & $0.018\,(02)$ & $0.018\,(02)$ \\
    &  & $0.01$ & $0.018\,(05)$ & $0.017\,(04)$ & $0.017\,(04)$ & $0.017\,(04)$ \\
    &  & $0.05$ & $0.023\,(05)$ & $0.017\,(05)$ & $0.017\,(05)$ & $0.018\,(05)$ \\
    & \multirow[t]{5}{*}{$20$} & $0.001$ & $0.023\,(03)$ & $0.023\,(04)$ & $0.023\,(04)$ & $0.023\,(03)$ \\
    &  & $0.01$ & $0.019\,(05)$ & $0.019\,(04)$ & $0.019\,(04)$ & $0.019\,(04)$ \\
    &  & $0.05$ & $0.024\,(05)$ & $0.018\,(05)$ & $0.018\,(05)$ & $0.018\,(05)$ \\
    \midrule
    \multirow[t]{10}{*}{$0.2$} & \multirow[t]{5}{*}{$10$} & $0.001$ & $0.019\,(03)$ & $0.019\,(02)$ & $0.019\,(02)$ & $0.018\,(02)$ \\
    &  & $0.01$ & $0.018\,(05)$ & $0.017\,(04)$ & $0.017\,(04)$ & $0.017\,(04)$ \\
    &  & $0.05$ & $0.023\,(05)$ & $0.018\,(05)$ & $0.018\,(05)$ & $0.018\,(05)$ \\
    & \multirow[t]{5}{*}{$20$} & $0.001$ & $0.023\,(03)$ & $0.023\,(03)$ & $0.023\,(03)$ & $0.023\,(03)$ \\
    &  & $0.01$ & $0.019\,(05)$ & $0.019\,(04)$ & $0.019\,(04)$ & $0.019\,(04)$ \\
    &  & $0.05$ & $0.024\,(05)$ & $0.018\,(04)$ & $0.018\,(04)$ & $0.018\,(04)$ \\
    \bottomrule
   \end{tabular}

%% file: assets/svm.tex
\begingroup%
  \makeatletter%
  \providecommand\color[2][]{%
    \errmessage{(Inkscape) Color is used for the text in Inkscape, but the package 'color.sty' is not loaded}%
    \renewcommand\color[2][]{}%
  }%
  \providecommand\transparent[1]{%
    \errmessage{(Inkscape) Transparency is used (non-zero) for the text in Inkscape, but the package 'transparent.sty' is not loaded}%
    \renewcommand\transparent[1]{}%
  }%
  \providecommand\rotatebox[2]{#2}%
  \newcommand*\fsize{\dimexpr\f@size pt\relax}%
  \newcommand*\lineheight[1]{\fontsize{\fsize}{#1\fsize}\selectfont}%
  \ifx\svgwidth\undefined%
    \setlength{\unitlength}{576bp}%
    \ifx\svgscale\undefined%
      \relax%
    \else%
      \setlength{\unitlength}{\unitlength * \real{\svgscale}}%
    \fi%
  \else%
    \setlength{\unitlength}{\svgwidth}%
  \fi%
  \global\let\svgwidth\undefined%
  \global\let\svgscale\undefined%
  \ifx\svgfont\undefined%
  \global\let\svgfont\footnotesize
  \fi%
  \makeatother%
  \begin{picture}(1,0.3125)%
    \lineheight{1}%
    \setlength\tabcolsep{0pt}%
    \put(0,0){\includegraphics[width=\unitlength,page=1]{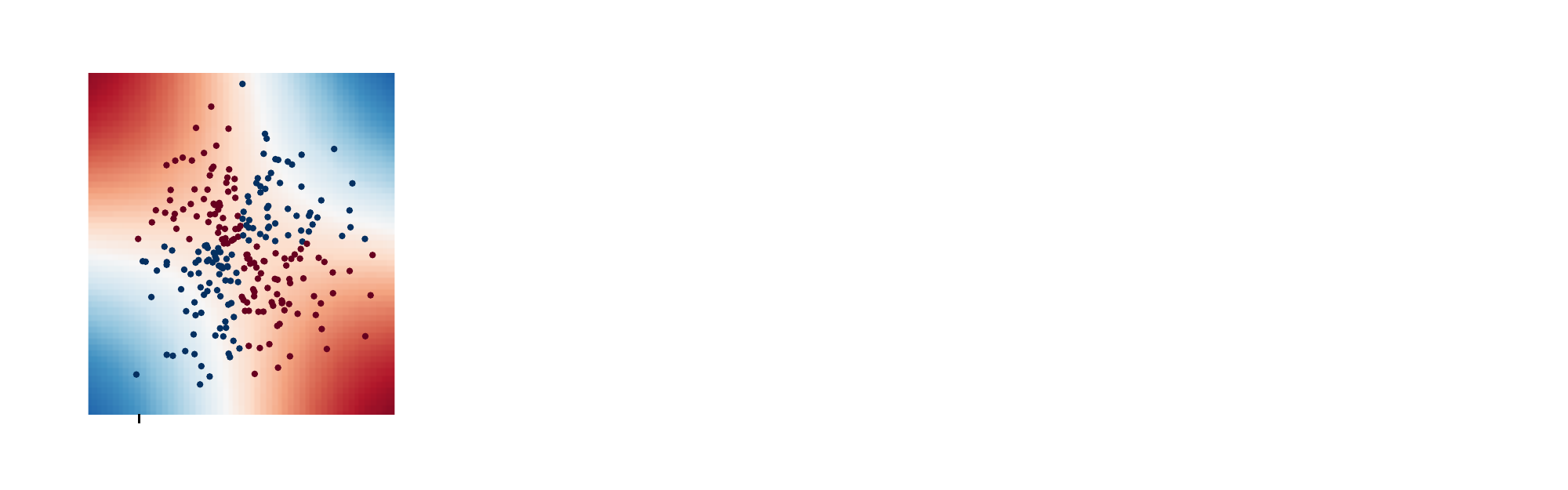}}%
    \put(0.08892572,0.015){\makebox(0,0)[t]{\lineheight{1.25}\smash{\begin{tabular}[t]{c}\svgfont{}-2\end{tabular}}}}%
    \put(0,0){\includegraphics[width=\unitlength,page=2]{assets/_svm.pdf}}%
    \put(0.15403411,0.015){\makebox(0,0)[t]{\lineheight{1.25}\smash{\begin{tabular}[t]{c}\svgfont{}0\end{tabular}}}}%
    \put(0,0){\includegraphics[width=\unitlength,page=3]{assets/_svm.pdf}}%
    \put(0.21914249,0.015){\makebox(0,0)[t]{\lineheight{1.25}\smash{\begin{tabular}[t]{c}\svgfont{}2\end{tabular}}}}%
    \put(0,0){\includegraphics[width=\unitlength,page=4]{assets/_svm.pdf}}%
    \put(0.04421875,0.07804534){\makebox(0,0)[rt]{\lineheight{1.25}\smash{\begin{tabular}[t]{r}\svgfont{}-2\end{tabular}}}}%
    \put(0,0){\includegraphics[width=\unitlength,page=5]{assets/_svm.pdf}}%
    \put(0.04421875,0.15052219){\makebox(0,0)[rt]{\lineheight{1.25}\smash{\begin{tabular}[t]{r}\svgfont{}0\end{tabular}}}}%
    \put(0,0){\includegraphics[width=\unitlength,page=6]{assets/_svm.pdf}}%
    \put(0.04421875,0.22299904){\makebox(0,0)[rt]{\lineheight{1.25}\smash{\begin{tabular}[t]{r}\svgfont{}2\end{tabular}}}}%
    \put(0,0){\includegraphics[width=\unitlength,page=7]{assets/_svm.pdf}}%
    \put(0.15403411,0.27625){\makebox(0,0)[t]{\lineheight{1.25}\smash{\begin{tabular}[t]{c}{}SAA\end{tabular}}}}%
    \put(0,0){\includegraphics[width=\unitlength,page=8]{assets/_svm.pdf}}%
    \put(0.41600906,0.015){\makebox(0,0)[t]{\lineheight{1.25}\smash{\begin{tabular}[t]{c}\svgfont{}-2\end{tabular}}}}%
    \put(0,0){\includegraphics[width=\unitlength,page=9]{assets/_svm.pdf}}%
    \put(0.48111746,0.015){\makebox(0,0)[t]{\lineheight{1.25}\smash{\begin{tabular}[t]{c}\svgfont{}0\end{tabular}}}}%
    \put(0,0){\includegraphics[width=\unitlength,page=10]{assets/_svm.pdf}}%
    \put(0.54622581,0.015){\makebox(0,0)[t]{\lineheight{1.25}\smash{\begin{tabular}[t]{c}\svgfont{}2\end{tabular}}}}%
    \put(0,0){\includegraphics[width=\unitlength,page=11]{assets/_svm.pdf}}%
    \put(0.37130207,0.07804534){\makebox(0,0)[rt]{\lineheight{1.25}\smash{\begin{tabular}[t]{r}\svgfont{}-2\end{tabular}}}}%
    \put(0,0){\includegraphics[width=\unitlength,page=12]{assets/_svm.pdf}}%
    \put(0.37130207,0.15052219){\makebox(0,0)[rt]{\lineheight{1.25}\smash{\begin{tabular}[t]{r}\svgfont{}0\end{tabular}}}}%
    \put(0,0){\includegraphics[width=\unitlength,page=13]{assets/_svm.pdf}}%
    \put(0.37130207,0.22299904){\makebox(0,0)[rt]{\lineheight{1.25}\smash{\begin{tabular}[t]{r}\svgfont{}2\end{tabular}}}}%
    \put(0,0){\includegraphics[width=\unitlength,page=14]{assets/_svm.pdf}}%
    \put(0.48111746,0.27625){\makebox(0,0)[t]{\lineheight{1.25}\smash{\begin{tabular}[t]{c}{}TV\end{tabular}}}}%
    \put(0,0){\includegraphics[width=\unitlength,page=15]{assets/_svm.pdf}}%
    \put(0.74309238,0.015){\makebox(0,0)[t]{\lineheight{1.25}\smash{\begin{tabular}[t]{c}\svgfont{}-2\end{tabular}}}}%
    \put(0,0){\includegraphics[width=\unitlength,page=16]{assets/_svm.pdf}}%
    \put(0.80820078,0.015){\makebox(0,0)[t]{\lineheight{1.25}\smash{\begin{tabular}[t]{c}\svgfont{}0\end{tabular}}}}%
    \put(0,0){\includegraphics[width=\unitlength,page=17]{assets/_svm.pdf}}%
    \put(0.87330919,0.015){\makebox(0,0)[t]{\lineheight{1.25}\smash{\begin{tabular}[t]{c}\svgfont{}2\end{tabular}}}}%
    \put(0,0){\includegraphics[width=\unitlength,page=18]{assets/_svm.pdf}}%
    \put(0.6983854,0.07804534){\makebox(0,0)[rt]{\lineheight{1.25}\smash{\begin{tabular}[t]{r}\svgfont{}-2\end{tabular}}}}%
    \put(0,0){\includegraphics[width=\unitlength,page=19]{assets/_svm.pdf}}%
    \put(0.6983854,0.15052219){\makebox(0,0)[rt]{\lineheight{1.25}\smash{\begin{tabular}[t]{r}\svgfont{}0\end{tabular}}}}%
    \put(0,0){\includegraphics[width=\unitlength,page=20]{assets/_svm.pdf}}%
    \put(0.6983854,0.22299904){\makebox(0,0)[rt]{\lineheight{1.25}\smash{\begin{tabular}[t]{r}\svgfont{}2\end{tabular}}}}%
    \put(0,0){\includegraphics[width=\unitlength,page=21]{assets/_svm.pdf}}%
    \put(0.80820078,0.27625){\makebox(0,0)[t]{\lineheight{1.25}\smash{\begin{tabular}[t]{c}{}$\bar{\CVAR}$\end{tabular}}}}%
    \put(0,0){\includegraphics[width=\unitlength,page=22]{assets/_svm.pdf}}%
    \put(0.28692881,0.07081707){\makebox(0,0)[lt]{\lineheight{1.25}\smash{\begin{tabular}[t]{l}\svgfont{}-10\end{tabular}}}}%
    \put(0,0){\includegraphics[width=\unitlength,page=23]{assets/_svm.pdf}}%
    \put(0.28692881,0.13394156){\makebox(0,0)[lt]{\lineheight{1.25}\smash{\begin{tabular}[t]{l}\svgfont{}0\end{tabular}}}}%
    \put(0,0){\includegraphics[width=\unitlength,page=24]{assets/_svm.pdf}}%
    \put(0.28692881,0.19706604){\makebox(0,0)[lt]{\lineheight{1.25}\smash{\begin{tabular}[t]{l}\svgfont{}10\end{tabular}}}}%
    \put(0,0){\includegraphics[width=\unitlength,page=25]{assets/_svm.pdf}}%
    \put(0.61401214,0.07997168){\makebox(0,0)[lt]{\lineheight{1.25}\smash{\begin{tabular}[t]{l}\svgfont{}-10\end{tabular}}}}%
    \put(0,0){\includegraphics[width=\unitlength,page=26]{assets/_svm.pdf}}%
    \put(0.61401214,0.1417839){\makebox(0,0)[lt]{\lineheight{1.25}\smash{\begin{tabular}[t]{l}\svgfont{}0\end{tabular}}}}%
    \put(0,0){\includegraphics[width=\unitlength,page=27]{assets/_svm.pdf}}%
    \put(0.61401214,0.20359611){\makebox(0,0)[lt]{\lineheight{1.25}\smash{\begin{tabular}[t]{l}\svgfont{}10\end{tabular}}}}%
    \put(0,0){\includegraphics[width=\unitlength,page=28]{assets/_svm.pdf}}%
    \put(0.94109546,0.08010896){\makebox(0,0)[lt]{\lineheight{1.25}\smash{\begin{tabular}[t]{l}\svgfont{}-10\end{tabular}}}}%
    \put(0,0){\includegraphics[width=\unitlength,page=29]{assets/_svm.pdf}}%
    \put(0.94109546,0.14185908){\makebox(0,0)[lt]{\lineheight{1.25}\smash{\begin{tabular}[t]{l}\svgfont{}0\end{tabular}}}}%
    \put(0,0){\includegraphics[width=\unitlength,page=30]{assets/_svm.pdf}}%
    \put(0.94109546,0.2036092){\makebox(0,0)[lt]{\lineheight{1.25}\smash{\begin{tabular}[t]{l}\svgfont{}10\end{tabular}}}}%
    \put(0,0){\includegraphics[width=\unitlength,page=31]{assets/_svm.pdf}}%
  \end{picture}%
\endgroup%

%% file: assets/svm-table.tex
\begin{tabular}{lllccc}
    \toprule
    $\delta$ & $\sigma$ & $\lambda$ & SAA & TV & $\bar{\CVAR}$ \\
   \midrule
   \multirow[t]{20}{*}{$0.05$} 
   & \multirow[t]{4}{*}{$0.05$} & $10^{4}$ & 0.474\,(0.052) & 0.126\,(0.100) & 0.111\,(0.054) \\
   &  & $10^{6}$ & 0.170\,(0.057) & 0.052\,(0.035) & 0.052\,(0.041) \\
   &  & $10^{8}$ & 0.053\,(0.020) & 0.020\,(0.008) & 0.022\,(0.011) \\
   &  & $10^{10}$ & 0.039\,(0.029) & 0.027\,(0.013) & 0.022\,(0.017) \\
   & \multirow[t]{4}{*}{$0.25$} & $10^{4}$ & 0.097\,(0.047) & 0.045\,(0.038) & 0.038\,(0.035) \\
   &  & $10^{6}$ & 0.042\,(0.012) & 0.019\,(0.004) & 0.016\,(0.003) \\
   &  & $10^{8}$ & \bftab 0.022\,(0.010) & 0.019\,(0.004) & 0.018\,(0.003) \\
   &  & $10^{10}$ & 0.067\,(0.082) & 0.061\,(0.114) & 0.043\,(0.045) \\
   & \multirow[t]{4}{*}{$0.5$} & $10^{4}$ & 0.055\,(0.021) & 0.023\,(0.007) & 0.023\,(0.006) \\
   &  & $10^{6}$ & 0.030\,(0.008) & 0.020\,(0.003) & 0.020\,(0.004) \\
   &  & $10^{8}$ & 0.023\,(0.008) & 0.024\,(0.007) & 0.024\,(0.007) \\
   &  & $10^{10}$ & 0.202\,(0.140) & 0.145\,(0.110) & 0.114\,(0.110) \\
   \midrule
   \multirow[t]{20}{*}{$0.1$} 
   & \multirow[t]{4}{*}{$0.05$} & $10^{4}$ & 0.474\,(0.052) & 0.154\,(0.093) & 0.145\,(0.075) \\
   &  & $10^{6}$ & 0.170\,(0.057) & 0.059\,(0.048) & 0.060\,(0.046) \\
   &  & $10^{8}$ & 0.053\,(0.020) & 0.022\,(0.011) & 0.021\,(0.012) \\
   &  & $10^{10}$ & 0.039\,(0.029) & 0.030\,(0.021) & 0.030\,(0.020) \\
   & \multirow[t]{4}{*}{$0.25$} & $10^{4}$ & 0.097\,(0.047) & 0.046\,(0.042) & 0.041\,(0.036) \\
   &  & $10^{6}$ & 0.042\,(0.012) & \bftab 0.015\,(0.003) & \bftab 0.015\,(0.003) \\
   &  & $10^{8}$ & \bftab 0.022\,(0.010) & 0.020\,(0.005) & 0.018\,(0.004) \\
   &  & $10^{10}$ & 0.067\,(0.082) & 0.076\,(0.119) & 0.037\,(0.039) \\
   & \multirow[t]{4}{*}{$0.5$} & $10^{4}$ & 0.055\,(0.021) & 0.025\,(0.009) & 0.024\,(0.009) \\
   &  & $10^{6}$ & 0.030\,(0.008) & 0.021\,(0.004) & 0.021\,(0.005) \\
   &  & $10^{8}$ & 0.023\,(0.008) & 0.024\,(0.007) & 0.024\,(0.007) \\
   &  & $10^{10}$ & 0.202\,(0.140) & 0.145\,(0.097) & 0.112\,(0.110) \\
   \midrule
   \multirow[t]{20}{*}{$0.2$} 
   & \multirow[t]{4}{*}{$0.05$} & $10^{4}$ & 0.474\,(0.052) & 0.225\,(0.142) & 0.197\,(0.100) \\
   &  & $10^{6}$ & 0.170\,(0.057) & 0.074\,(0.050) & 0.075\,(0.051) \\
   &  & $10^{8}$ & 0.053\,(0.020) & 0.025\,(0.015) & 0.026\,(0.015) \\
   &  & $10^{10}$ & 0.039\,(0.029) & 0.033\,(0.023) & 0.036\,(0.026) \\
   & \multirow[t]{4}{*}{$0.25$} & $10^{4}$ & 0.097\,(0.047) & 0.049\,(0.044) & 0.045\,(0.040) \\
   &  & $10^{6}$ & 0.042\,(0.012) & 0.019\,(0.005) & 0.016\,(0.003) \\
   &  & $10^{8}$ & \bftab 0.022\,(0.010) & 0.017\,(0.004) & 0.017\,(0.004) \\
   &  & $10^{10}$ & 0.067\,(0.082) & 0.070\,(0.113) & 0.039\,(0.033) \\
   & \multirow[t]{4}{*}{$0.5$} & $10^{4}$ & 0.055\,(0.021) & 0.027\,(0.014) & 0.027\,(0.015) \\
   &  & $10^{6}$ & 0.030\,(0.008) & 0.021\,(0.005) & 0.020\,(0.005) \\
   &  & $10^{8}$ & 0.023\,(0.008) & 0.024\,(0.007) & 0.024\,(0.007) \\
   &  & $10^{10}$ & 0.202\,(0.140) & 0.186\,(0.144) & 0.112\,(0.101) \\
   \bottomrule
\end{tabular}
   

%% file: assets/svm-complexity.tex
\begin{tikzpicture}
    \definecolor{darkorange_}{RGB}{255,127,14}
    \definecolor{forestgreen_}{RGB}{44,160,44}
    \definecolor{steelblue_}{RGB}{31,119,180}
    
    \begin{axis}[
        numeric axis,
        xmode=log,
        xmin=10, xmax=1000,
        ymin=0, ymax=0.5,
        legend style={at={(axis cs: 1000, 0.6)}, anchor=north east, draw=none, yshift=-6pt, xshift=-6pt},
        legend cell align={left},
        xlabel={$n$},
        ylabel={misclass.\ rate},
        width=\columnwidth,
        height=0.375\columnwidth,
        darkorange/.style={darkorange_, dash dot, thick},
        forestgreen/.style={forestgreen_, dashed, thick},
        steelblue/.style={steelblue_, thick},
    ]
    \addlegendimage{steelblue}
    \addlegendentry{\footnotesize{}SAA};
    \addlegendimage{darkorange}
    \addlegendentry{\footnotesize{}TV};
    \addlegendimage{forestgreen}
    \addlegendentry{\footnotesize{}$\bar{\CVAR}$};

    \addplot [draw=steelblue_, fill=steelblue_, mark=-, only marks]
    table{%
    x  y
    10 0.18885
    10 0.39347
    };
    \addplot [draw=steelblue_, fill=steelblue_, mark=-, only marks]
    table{%
    x  y
    15 0.19874
    15 0.31073
    };
    \addplot [draw=steelblue_, fill=steelblue_, mark=-, only marks]
    table{%
    x  y
    23 0.16067
    23 0.24473
    };
    \addplot [draw=steelblue_, fill=steelblue_, mark=-, only marks]
    table{%
    x  y
    35 0.12376
    35 0.19851
    };
    \addplot [draw=steelblue_, fill=steelblue_, mark=-, only marks]
    table{%
    x  y
    53 0.11064
    53 0.18576
    };
    \addplot [draw=steelblue_, fill=steelblue_, mark=-, only marks]
    table{%
    x  y
    81 0.09463
    81 0.17309
    };
    \addplot [draw=steelblue_, fill=steelblue_, mark=-, only marks]
    table{%
    x  y
    123 0.07884
    123 0.15004
    };
    \addplot [draw=steelblue_, fill=steelblue_, mark=-, only marks]
    table{%
    x  y
    187 0.07128
    187 0.13723
    };
    \addplot [draw=steelblue_, fill=steelblue_, mark=-, only marks]
    table{%
    x  y
    284 0.05175
    284 0.11733
    };
    \addplot [draw=steelblue_, fill=steelblue_, mark=-, only marks]
    table{%
    x  y
    432 0.05331
    432 0.10888
    };
    \addplot [draw=steelblue_, fill=steelblue_, mark=-, only marks]
    table{%
    x  y
    657 0.04509
    657 0.09446
    };
    \addplot [draw=steelblue_, fill=steelblue_, mark=-, only marks]
    table{%
    x  y
    1000 0.04217
    1000 0.0867
    };
    \addplot [draw=darkorange_, fill=darkorange_, mark=-, only marks]
    table{%
    x  y
    10 0.16592
    10 0.39681
    };
    \addplot [draw=darkorange_, fill=darkorange_, mark=-, only marks]
    table{%
    x  y
    15 0.13697
    15 0.28749
    };
    \addplot [draw=darkorange_, fill=darkorange_, mark=-, only marks]
    table{%
    x  y
    23 0.10255
    23 0.17783
    };
    \addplot [draw=darkorange_, fill=darkorange_, mark=-, only marks]
    table{%
    x  y
    35 0.07374
    35 0.13799
    };
    \addplot [draw=darkorange_, fill=darkorange_, mark=-, only marks]
    table{%
    x  y
    53 0.04261
    53 0.10511
    };
    \addplot [draw=darkorange_, fill=darkorange_, mark=-, only marks]
    table{%
    x  y
    81 0.03782
    81 0.09301
    };
    \addplot [draw=darkorange_, fill=darkorange_, mark=-, only marks]
    table{%
    x  y
    123 0.02718
    123 0.06968
    };
    \addplot [draw=darkorange_, fill=darkorange_, mark=-, only marks]
    table{%
    x  y
    187 0.02711
    187 0.0468
    };
    \addplot [draw=darkorange_, fill=darkorange_, mark=-, only marks]
    table{%
    x  y
    284 0.0188
    284 0.04456
    };
    \addplot [draw=darkorange_, fill=darkorange_, mark=-, only marks]
    table{%
    x  y
    432 0.01779
    432 0.04326
    };
    \addplot [draw=darkorange_, fill=darkorange_, mark=-, only marks]
    table{%
    x  y
    657 0.00834000000000001
    657 0.02489
    };
    \addplot [draw=darkorange_, fill=darkorange_, mark=-, only marks]
    table{%
    x  y
    1000 0.00788999999999995
    1000 0.02727
    };
    \addplot [draw=forestgreen_, fill=forestgreen_, mark=-, only marks]
    table{%
    x  y
    10 0.16397
    10 0.39681
    };
    \addplot [draw=forestgreen_, fill=forestgreen_, mark=-, only marks]
    table{%
    x  y
    15 0.13622
    15 0.28749
    };
    \addplot [draw=forestgreen_, fill=forestgreen_, mark=-, only marks]
    table{%
    x  y
    23 0.11745
    23 0.18719
    };
    \addplot [draw=forestgreen_, fill=forestgreen_, mark=-, only marks]
    table{%
    x  y
    35 0.07229
    35 0.1427
    };
    \addplot [draw=forestgreen_, fill=forestgreen_, mark=-, only marks]
    table{%
    x  y
    53 0.04505
    53 0.10902
    };
    \addplot [draw=forestgreen_, fill=forestgreen_, mark=-, only marks]
    table{%
    x  y
    81 0.03645
    81 0.09358
    };
    \addplot [draw=forestgreen_, fill=forestgreen_, mark=-, only marks]
    table{%
    x  y
    123 0.02934
    123 0.07226
    };
    \addplot [draw=forestgreen_, fill=forestgreen_, mark=-, only marks]
    table{%
    x  y
    187 0.02444
    187 0.05278
    };
    \addplot [draw=forestgreen_, fill=forestgreen_, mark=-, only marks]
    table{%
    x  y
    284 0.01676
    284 0.0448
    };
    \addplot [draw=forestgreen_, fill=forestgreen_, mark=-, only marks]
    table{%
    x  y
    432 0.01941
    432 0.04194
    };
    \addplot [draw=forestgreen_, fill=forestgreen_, mark=-, only marks]
    table{%
    x  y
    657 0.00814000000000004
    657 0.02967
    };
    \addplot [draw=forestgreen_, fill=forestgreen_, mark=-, only marks]
    table{%
    x  y
    1000 0.00602999999999998
    1000 0.03102
    };
    \addplot [steelblue]
    table {%
    10 0.18885
    10 0.39347
    };
    \addplot [steelblue]
    table {%
    15 0.19874
    15 0.31073
    };
    \addplot [steelblue]
    table {%
    23 0.16067
    23 0.24473
    };
    \addplot [steelblue]
    table {%
    35 0.12376
    35 0.19851
    };
    \addplot [steelblue]
    table {%
    53 0.11064
    53 0.18576
    };
    \addplot [steelblue]
    table {%
    81 0.09463
    81 0.17309
    };
    \addplot [steelblue]
    table {%
    123 0.07884
    123 0.15004
    };
    \addplot [steelblue]
    table {%
    187 0.07128
    187 0.13723
    };
    \addplot [steelblue]
    table {%
    284 0.05175
    284 0.11733
    };
    \addplot [steelblue]
    table {%
    432 0.05331
    432 0.10888
    };
    \addplot [steelblue]
    table {%
    657 0.04509
    657 0.09446
    };
    \addplot [steelblue]
    table {%
    1000 0.04217
    1000 0.0867
    };
    \addplot [steelblue]
    table {%
    10 0.289249666666667
    15 0.245928666666667
    23 0.193936333333333
    35 0.166404
    53 0.143886
    81 0.131428333333333
    123 0.113370666666667
    187 0.0979253333333333
    284 0.081105
    432 0.077855
    657 0.0665683333333333
    1000 0.0620413333333333
    };
    \addplot [darkorange]
    table {%
    10 0.16592
    10 0.39681
    };
    \addplot [darkorange]
    table {%
    15 0.13697
    15 0.28749
    };
    \addplot [darkorange]
    table {%
    23 0.10255
    23 0.17783
    };
    \addplot [darkorange]
    table {%
    35 0.07374
    35 0.13799
    };
    \addplot [darkorange]
    table {%
    53 0.04261
    53 0.10511
    };
    \addplot [darkorange]
    table {%
    81 0.03782
    81 0.09301
    };
    \addplot [darkorange]
    table {%
    123 0.02718
    123 0.06968
    };
    \addplot [darkorange]
    table {%
    187 0.02711
    187 0.0468
    };
    \addplot [darkorange]
    table {%
    284 0.0188
    284 0.04456
    };
    \addplot [darkorange]
    table {%
    432 0.01779
    432 0.04326
    };
    \addplot [darkorange]
    table {%
    657 0.00834000000000001
    657 0.02489
    };
    \addplot [darkorange]
    table {%
    1000 0.00788999999999995
    1000 0.02727
    };
    \addplot [darkorange]
    table {%
    10 0.273421
    15 0.202971666666667
    23 0.145454333333333
    35 0.10497
    53 0.0682196666666667
    81 0.0654166666666667
    123 0.049121
    187 0.0404373333333333
    284 0.0333993333333333
    432 0.0311653333333333
    657 0.0191353333333333
    1000 0.0194006666666667
    };
    \addplot [forestgreen]
    table {%
    10 0.16397
    10 0.39681
    };
    \addplot [forestgreen]
    table {%
    15 0.13622
    15 0.28749
    };
    \addplot [forestgreen]
    table {%
    23 0.11745
    23 0.18719
    };
    \addplot [forestgreen]
    table {%
    35 0.07229
    35 0.1427
    };
    \addplot [forestgreen]
    table {%
    53 0.04505
    53 0.10902
    };
    \addplot [forestgreen]
    table {%
    81 0.03645
    81 0.09358
    };
    \addplot [forestgreen]
    table {%
    123 0.02934
    123 0.07226
    };
    \addplot [forestgreen]
    table {%
    187 0.02444
    187 0.05278
    };
    \addplot [forestgreen]
    table {%
    284 0.01676
    284 0.0448
    };
    \addplot [forestgreen]
    table {%
    432 0.01941
    432 0.04194
    };
    \addplot [forestgreen]
    table {%
    657 0.00814000000000004
    657 0.02967
    };
    \addplot [forestgreen]
    table {%
    1000 0.00602999999999998
    1000 0.03102
    };
    \addplot [forestgreen]
    table {%
    10 0.274501666666667
    15 0.202802333333333
    23 0.149837666666667
    35 0.105479333333333
    53 0.0759403333333333
    81 0.0632976666666667
    123 0.0510246666666667
    187 0.0385383333333333
    284 0.0336413333333333
    432 0.030228
    657 0.0202433333333333
    1000 0.0183866666666667
    };
    \end{axis}
    
    \end{tikzpicture}
    

%% file: content/preliminaries.tex
\section{Preliminaries}
\subsection{Monotone Cone}
Let $\Re^n_{\uparrow} \dfn \{x \in \Re^n \colon x_1 \leq x_2 \leq \dots \leq x_n\}$ 
denote the monotone cone. This cone and its polar have a history in isotonic regression \cite{Barlow1972}
and majorization \cite{Steerneman1990}. 

We show that $\major{n}$ and $\Re_\uparrow^n$ are related.
\begin{lemma} \label{prop:dual-monotone-cone}
    Let $\Re_\uparrow^n$ be the monotone cone and $\major{n}$ as in \cref{eq:majorization-cone}.
    Then $\major{n}$ is the polar of $\Re_\uparrow^n$.
\end{lemma}
\begin{proof}
    The monotone cone is polyhedral with $\Re_\uparrow^n = \{x \colon Mx \leq 0\}$ for 
    $M \in \Re^{n-1 \times n}$ with $Mx = (x_1 - x_2, x_2 - x_3, \dots, x_{n-1} - x_n)$. 
    The definition of the polar cone is thus 
    \begin{align*}
        (\Re_\uparrow^n)^\circ &= \left\{ y \in \Re^n \colon \<x, y\> \leq 0, \forall x \text{ s.t. } M x \leq 0 \right\}.
    \end{align*}
    By Farkas' lemma \cite[p.~263]{Boyd2004} we have either $Mx \leq 0$ and $\<x, y\> > 0$ or $\trans{M} \lambda = y$ and $\lambda \geq 0$. 
    So
    \begin{align*}
        (\Re_\uparrow^n)^\circ = \left\{ y \in \Re^n \colon y = \trans{M} \lambda, \lambda \geq 0 \right\}.
    \end{align*}
    Note that $\<\lambda, Mx\> = \sum_{i=1}^{n-1} \lambda_i (x_i - x_{i+1}) = \sum_{j=1}^{n} x_j (\lambda_{j} - \lambda_{j-1}) = \<\trans{M} \lambda, x\>$,
    where $\lambda_0 = \lambda_{n} = 0$. Thus $y \in (\Re_\uparrow^n)^\circ$ iff $y = \trans{M} \lambda$ for $\lambda \geq 0$. Here $y = \trans{M} \lambda$ holds iff 
    \begin{align*}
        y_j &= \lambda_j - \lambda_{j-1}, &&\forall j \in [n].  \\
        \Leftrightarrow \quad \ssum_{j=1}^k y_j &= \ssum_{j=1}^k \lambda_j - \lambda_{j-1} = \lambda_k, &&\forall k \in [n]. 
    \end{align*}
    Since $\lambda \geq 0$ we have $\sum_{j=1}^k y_j = \lambda_k \geq 0$ for $k \in [n-1]$ and $\sum_{j=1}^{n} y_j = \lambda_n = 0$. 
    These are the constraints in \cref{eq:majorization-cone}. 
\end{proof}


\begin{arxiv}
\subsection{Details on $\bar{\CVAR}$} \label{app:cvar}
We first characterize the conditional value-at-risk in terms of order statistics as a \emph{distortion risk} below.
\begin{lemma} \label{lem:cvar-rewritten}
    Consider $\CVAR_{n}^{\gamma} \colon \Re^{n} \to \Re$ as
    \begin{equation} \label{eq:cvar-definition}
        \CVAR_{n}^{\gamma}[X] = \inf_{\tau} \left\{ \tau + \frac{1}{(1-\gamma) n} \sum_{i=1}^{n} [X_i - \tau]_+ \right\},
    \end{equation}        
    for $\gamma \in [0, 1]$. Then
    \begin{equation*}
        (1 - \gamma)\CVAR_n^{\gamma}[X] = \left( \frac{d}{n} - \gamma \right) X_{(d)} + \sum_{i=d+1}^{n} \frac{X_{(i)}}{n},
    \end{equation*}
    with $d \dfn \lceil n\gamma \rceil$. So $\CVAR_n^\gamma$ is a distortion risk. 
\end{lemma}
\begin{proof}
    Consider the minimizers in the definition of $\CVAR$:
    \begin{equation*}
        \argmin_\tau \left\{ \tau + \frac{1}{(1-\gamma) n} \sum_{i=1}^{n} [X_i - \tau]_+ \right\}.
    \end{equation*} 
    By \cite[Thm.~1]{Rockafellar2000}, this set is a closed bounded interval with the left endpoint being
    \begin{align*}
        \VAR_{n}^{\gamma}[X] &\dfn \inf_x \{x \colon F_n(x) \geq \gamma\} \\
                            &= \inf_x \left\{ x \colon \sum_{i=1}^{n }\bm{1}_{(-\infty, x]}(X_i) \geq \gamma n \right\} = X_{(d)},
    \end{align*}
    with $F_n$ the empirical cdf, the definition of which we plugged in for the second equality.
    For the third equality note that the left-hand side counts the number of values $X_i$ smaller than or equal to $x$. 
    Assume 
    \begin{equation}
        X_{(d-k-1)} < X_{(d-k)} = X_{(d-k+1)} = \dots = X_{(d)},
    \end{equation}
    for $k \geq 0$. Then clearly there are at least $d = \lceil n \gamma \rceil > n\gamma$ values smaller than or equal to $X_{(d-k)}$. 
    For any $z < X_{(d-k)}$ 
    there are at most $d-k-1$ samples values than or equal. Hence 
    $\VAR_{n}^{\gamma}[X] = X_{(d-k)} = X_{(d)}$. 

    Plugging into the cost of \cref{eq:cvar-definition} gives 
    \begin{align}
        &X_{(d)} + \frac{1}{(1-\gamma) n} \sum_{i=1}^{n} [X_i - X_{(d)}]_+ \nonumber \\
        &\quad = X_{(d)} + \frac{1}{(1-\gamma) n} \sum_{i=d + 1}^{n} (X_{(i)} - X_{(d)}),\label{eq:cvar-rewritten}
    \end{align}
    where we used $X_{(i)} \leq X_{(d)}$ for $i \leq d$. The stated result follows from some 
    basic algebraic manipulation. Finally note that $d/n - \gamma = (\lceil n\gamma \rceil - n\gamma)/n \leq 1/n$. 
    So the monotonicity constraint on $\mu$ in \cref{cor:simple-distortion} is also satisfied.
\end{proof}

We can then rewrite \cref{eq:biased-cvar} as 
\begin{equation*}
    \bar{\CVAR}_n^\gamma[X] \dfn (1 - \gamma) \CVAR_{n-1}^{\gamma}[(X_{(i)})_{i=1}^{n-1}] + \gamma X_{(n)}.
\end{equation*}
The associated distortion risk has weights $\mu^{(\gamma)}$ as in \cref{eq:cvar-weights}.
So we need $\gamma \geq 1/(n-1)$ for $\mu \in \Delta^n \cap \Re_\uparrow^n$ to hold. The advantage of $\bar{\CVAR}$ 
as a well-calibrated risk measure over $\CVAR$ is that additional weight is placed on the largest sample. 
This often makes the mean bound associated with $\bar{\CVAR}$ less conservative compared to $\CVAR$ 
for the same confidence level. 
\end{arxiv}

%% file: content-old/appendix/svm.tex
\section{Support Vector Machines} \label{app:svm}
\newcommand{\threshold}{b}
Let $\set{H}$ be some \emph{reproducing kernel Hilbert Space (RKHS)} \cite[Def.~2.9]{Scholkopf2002} with reproducing kernel $\kappa \colon \Re^d \times \Re^d \to \Re$ for some $d \in \N$. 
Here $\<f, g\>_{\set{H}}$ denotes the inner product associated with $\set{H}$ and $\nrm{f}_{\set{H}}^2 = \<f, f\>$ for $f, g \in \set{H}$. 
The \emph{primal problem} for learning a support vector machine is usually given in terms of the \emph{hinge loss}:
\begin{equation*}
    \begin{alignedat}{2}
        &\minimize_{(f, \threshold) \in \set{H} \times \Re} &\quad & \frac{1}{2} \nrm{f}_{\set{H}}^2 + \lambda \E\left[ 1 - Y (f(X) - \threshold) \right]_+.
    \end{alignedat}
\end{equation*}
with $\lambda > 0$, $X \colon \Omega \to \Re^d$ and $Y \colon \Omega \to \{-1, 1\}$. Using the reproducing property of $\kappa$ (cf. \cite[Def.~2.9.1]{Scholkopf2002}) we have $f(X) = \<\kappa(X, \cdot), f\>$. 
Given a sample $\{(X_i, Y_i)\}_{i=1}^{n}$, let $\hat{Y} = \mathrm{diag}(Y_1, \dots, Y_n) \in \Re^{n \times n}$ and $\hat{K} \colon \set{H} \to \Re^n$ 
a linear operator such that $(\hat{K} f)_{i} = \<\kappa(X_i, \cdot), f\>$. We do not include a data-point modeling the $\esssup$ 
of the random loss. Instead the largest sample will act as a replacement. 

We then replace the expectation with a proxy cost, as in \cref{eq:ordered-risk}. 
The robustified, data-driven problem then becomes
\begin{equation} \label{eq:primal-problem}
    \begin{alignedat}{2}
        &\minimize_{(f, \threshold) \in \set{H} \times \Re} &\quad & \frac{1}{2} \nrm{f}_{\set{H}}^2 + \lambda \rho\left[ \one_n - \hat{Y} \left(\hat{K} f - \threshold \one_n \right) \right]_+,
    \end{alignedat}
\end{equation}
where $\rho(x) = \sup_{\mu \in \amb} \, \<\mu, x\>$ is a support function.

\begin{proposition} \label{prop:svm-reformulation}
    The value of \cref{eq:primal-problem} equals
    \begin{equation} \label{eq:svm-dual-problem}
        \begin{alignedat}{2}
            &\maximize_{(\alpha, \beta) \in \Re^n_{+} \times \Re^n_+} &\qquad & \sum_{i=1}^{n} \alpha_i - \frac{1}{2} \sum_{i, j=1}^{n} \alpha_i \alpha_j Y_i Y_j \kappa(X_i, X_j) \\
            &\stt && \sum_{i=1}^{n} \alpha_i Y_i = 0, \, \frac{\alpha + \beta}{\lambda} \in \amb.
        \end{alignedat}
    \end{equation}
    Moreover, let $\alpha^\star, \beta^\star$ denote the optimizers and $\set{J} \dfn \{j \in [n] \colon \alpha_j > 0, \, \beta_j > 0\}$.
    Then
    \begin{align*}
        f^\star &= \sum_{i=1}^{n} \alpha_i^\star Y_i \kappa(X_i, \cdot) \\
        b^\star &= \sum_{i=1}^{n} \alpha_i^\star Y_i \kappa(X_i, X_j) - Y_j, \quad \forall j \in \set{J} \, \text{when } \set{J} \neq \emptyset
    \end{align*}
    are the optimizers\footnote{In practice, we pick $b^\star$ as the average of the values over all $j \in \set{J}$.} of \cref{eq:primal-problem} when $\set{J} \neq \emptyset$.
    When $\set{J} = \emptyset$, then $b^\star$ can be determined by solving \cref{eq:primal-problem}, keeping $f = f^\star$ fixed.
\end{proposition}
\begin{proof}
    We write \cref{eq:primal-problem} with slack variables first 
    \begin{equation*}
        \begin{alignedat}{2}
            &\minimize_{(f, \threshold, s) \in \set{H} \times \Re \times \Re^n} & \qquad & \frac{1}{2} \nrm{f}_{\set{H}}^2 + \lambda\rho(s) \\
            &\stt && \hat{Y} (\hat{K} f - \one_n \threshold) - \one_n + s \geq 0 \\
            &&& s \geq 0.
        \end{alignedat}
    \end{equation*}

    The next step is to apply Lagrangian duality over Hilbert spaces as presented in \cite[Prop.~19.18]{Bauschke2011}.   
    We first bring the problem in the standard form:
    \begin{equation*}
        \minimize_{x \in \set{G}} \qquad h(x) + g(L x),
    \end{equation*}
    with $\set{G} \dfn \set{H} \times \Re \times \Re^n$ a Hilbert space, elements of which we partition as $x = (f, \threshold, s)$.
    Let $L \colon \set{G} \to \Re^{2n}$ denote the linear operator defined as 
    \begin{equation} \label{eq:ell-definition}
        L x = \left(\hat{Y} \hat{K} f - \hat{Y} \one_n \threshold + s, s\right). 
    \end{equation}
    Its adjoint -- defined implicitly as $L^* \colon \Re^{2n} \to \set{G}$ such that $\<Lx, v\> = \<x, L^* v\>$ -- is given as
    \begin{equation*}
        L^* v = (\hat{K}^* \hat{Y} \alpha,  -\one_n^T \hat{Y} \alpha, \alpha + \beta),
    \end{equation*}
    for $v = (\alpha, \beta)$ and 
    $\hat{K}^* \hat{Y} \alpha = \sum_{i=1}^{n} \alpha_i Y_i \kappa(X_i, \cdot)$. 
    The functions $h \colon \set{G} \to \eRe$ and $g \colon \Re^n \to \eRe$ are given as 
    \begin{align*}
        h(x) &\dfn  \nrm{f}_{\set{H}}^2/2 + \lambda\rho(s)\\
        g(Lx)&\dfn \indi_{\Re^{2n}_+}(Lx - \delta),
    \end{align*}
    with $\delta = (\one_n, 0) \in \Re^{2n}$. 

    First we prove that strong duality holds, a sufficient condition for which is $\mathrm{int} (\dom g) \cap L \dom h$ (cf. \cite[Thm.~15.23, Prop.~6.19(vii)]{Bauschke2011}).
    Note that $\dom g = \Re^{2n}_+ + \delta$ and $\dom h = \set{H} \times \Re \times \Re^n$, keeping in mind that $\dom \rho = \Re^n$ 
    since $\rho$ is coherent. Note that $0 \in \set{H}$ and $0 \in \Re$ and that $L(0, 0, s) = (s, s)$.
    So $\distort = \{(s, s) \colon s \in \Re^n\} \subset L \dom h$. Hence $\mathrm{int} (\dom g) \cap L \dom h \supset (\mathrm{int} (\Re^{2n}_+) + \delta) \cap \distort \neq \emptyset$.

    Consider the convex conjugates $h^*$ and $g^*$. Then
    \begin{align*}
        h^*(\bar{x}) = \nrm{f}_{\set{H}}^2/2 + \indi_{\{0\}}(\bar{\threshold}) + \lambda \rho^*(\bar{s}/\lambda),
    \end{align*}
    where we used the definition of the convex conjugate, the partitioning $\bar{x} = (\bar{f}, \bar{\threshold}, \bar{s}) \in \set{G}$,
    \cite[Prop.~13.16, Prop.~13.20(i)]{Bauschke2011} and seperability of $h$. 
    Since $\rho$ is a support function we have $\rho^* = \indi_{\amb}$. 
    Also, again letting $v = (\alpha, \beta)$,
    \begin{equation*}
        g^*(v) = \indi_{\Re^{2n}_-}(v) + \trans{\one_n} \alpha
    \end{equation*}
    by \cite[Prop.~13.20(ii)]{Bauschke2011}. The dual problem, the value of which equals minus the primal by strong duality, is then given as \cite[Prop.~19.18]{Bauschke2011}:
    \begin{equation*}
        \begin{alignedat}{2}
            &\minimize_{(\alpha, \beta) \in \Re^n \times \Re^n} &\qquad& \frac{1}{2} \nrm{\hat{K}^* \hat{Y} \alpha}_{\set{H}}^2 - \trans{\one_n} \alpha \\
            &\stt && \trans{\one_n} \hat{Y} \alpha = 0, \, (\alpha + \beta)/\lambda \in \amb, \, (\alpha, \beta) \geq 0
        \end{alignedat}
    \end{equation*}
    where we already integrated the indicator functions in the constraints. 
    After adding the minus sign, this is equivalent to the problem in the theorem. 

    We next compute the subgradients. Note that,
    \begin{align} \label{eq:subgradient-h}
        \partial h^*(\bar{x}) = \{\bar{f}\} \times (\Re \cap \{\bar{\threshold}\}^\bot) \times \mathcal{N}_{\amb}(\bar{s}/\lambda)/\lambda,
    \end{align}
    with $\mathcal{N}_{\set{C}}$ denotes the \emph{normal cone} for set $\set{C}$ \cite[Def.~6.37]{Bauschke2011} and 
    $\set{C}^\bot = \{u \colon \<x, u\> = 0, \forall x \in \set{C}\}$ denotes the \emph{orthogonal complement} of $\set{C}$. 
    The first term follows from differentiability. The second term follows from \cite[Ex.~16.12, Ex.~6.39]{Bauschke2011}
    which states $\partial \indi_{\{0\}}(\bar{\threshold}) = \mathcal{N}_{\{0\}}(\bar{\threshold}) = \Re^n \cap \{\bar{\threshold}\}^{\bot}$. 
    The third follows from applying \cite[Ex.~16.12, Cor.~16.42]{Bauschke2011} to $\indi_{\amb} \circ \lambda^{-1} I_n$. 
    Similarly 
    \begin{align}
        \partial g^*(v) &= \Re^n_+ \cap \{v\}^\bot + \delta \nonumber \\
                             &= \{u + \delta \colon u \geq 0, \, \trans{u} v = 0\}, \label{eq:subgradient-g}
    \end{align}
    where we apply \cite[Cor.~16.38]{Bauschke2011}, differentiability of the second term and again \cite[Ex.~16.12, Ex.~6.39]{Bauschke2011}
    to deal with the indicator. 

    By \cite[Prop.~19.17(v), Prop.~19.18(v)]{Bauschke2011} the optimizers $v^\star = (\alpha^\star, \beta^\star)$ and $
    x^\star = (f^\star, \threshold^\star, s^\star)$
    must satisfy
    \begin{align*}
        && L^* \alpha^\star &\in \partial h(x^\star) &&\text{and}& -\alpha^\star &\in \partial g(L x^\star) \\
        &\Leftrightarrow & x^\star &\in \partial h^*(L^* \alpha^\star) &&\text{and}& L x^\star &\in \partial g^*(-\alpha^\star),
    \end{align*}
    where we used \cite[Cor.~16.24]{Bauschke2011} to express the optimality conditions in terms of the subgradients of the conjugates. 
    Plugging in the subgradient determined in \cref{eq:subgradient-h} gives 
    \begin{align*}
        f^* = \hat{K}^* \hat{Y} \alpha^\star = \sum_{i=1}^{n} \alpha^\star_i Y_i \kappa(X_i, \cdot). 
    \end{align*}
    Comparing with \cite[Eq.~7.25]{Scholkopf2002} shows that this is the usual SVM solution. 
    For the threshold $\threshold^\star$ we use \cref{eq:subgradient-g}, giving
    \begin{align*}
        &&&L x^\star \in \partial g^*(-\alpha^\star) \\
        &\Leftrightarrow && \sum_{i=1}^{n} \alpha_i^\star \left( Y_i (f^\star(X_i) - \threshold^\star) - 1 + s_i^\star \right) + \sum_{i=1}^{n} \beta_i^\star s_i^\star = 0,\\
        &&& Y_i (f^\star(X_i) - \threshold^\star) - 1 + s_i \geq 0, \quad \forall i \in [n],\\
        &&& s^\star \geq 0.
    \end{align*}
    Let $\set{J} = \{j \in [n] \colon \alpha_j^\star > 0, \, \beta_j^\star > 0\}$. From the above conditions we have 
    $s_j^\star = 0$ for all $j \in \set{J}$. Similarly we require 
    \begin{equation*}
        Y_j (f^\star(X_j) - \threshold^\star) - 1 = 0, \quad \forall j \in \set{J}.
    \end{equation*}
    Using $Y_j \in \{-1, 1\}$ to mulyiply both sides with $Y_j$ and the characterization of $f^\star$ above results in 
    \begin{equation*}
        b^\star = \sum_{i=1}^n \alpha^\star_i Y_i \kappa(X_i, X_j) - Y_j, \quad \forall j \in \set{J}.
    \end{equation*}
    We can again compare with the classical SVM setting (cf. \cite[Eq.~47.32]{Scholkopf2002} and the discussion at \cite[p.~206]{Scholkopf2002}),
    seeing that the condition is similar. When $\set{J} = \emptyset$, we cannot generate trivial constraints on $b^\star$. In that case, we note the $f^\star$
    is still a valid minimizer, thus $b^\star$ must minimize \cref{eq:primal-problem} for $f = f^\star$.
\end{proof}

For $\bar{\CVAR}$ as in \cref{eq:biased-cvar} we can characterize the ambiguity set $\amb$ efficiently in terms of a polyhedral set as in \cite{Bertsimas2009b}.
So \cref{eq:svm-dual-problem} is a quadratic program. For divergence-based risk measures the ambiguity set $\amb_{\alpha}$ in \cref{eq:ambiguity-phi-div} 
is convex so we can implement it directly. It is polyhedral for the total variation.